\documentclass{article}

\usepackage{geometry}
\usepackage{graphicx} % figure environment
\usepackage{epstopdf} % eps

\usepackage[colorlinks, hypertexnames=false]{hyperref}
\hypersetup{linkcolor=blue, citecolor=red} % set color for references

\usepackage{bm}
\usepackage{amsmath} % math
\usepackage{amssymb} % math symbols
\usepackage{stmaryrd} % jump symbol
\usepackage{multirow}

\usepackage{amsthm} % theorems
\usepackage{mathtools} % redundant label removal, adjust limits
\usepackage[dvipsnames]{xcolor}
\definecolor{red}{rgb}{1.00,0.00,0.00}
{\numberwithin{equation}{section}
\setlength{\parindent}{1em}

\newtheorem{theorem}{Theorem}[section]
\newtheorem{lemma}{Lemma}[section]
\newtheorem{remark}{Remark}[section]

\renewcommand{\footnotesize}{\scriptsize}
\newcommand{\normmm}[1]{{\left\vert\kern-0.25ex\left\vert
\kern-0.25ex\left\vert #1
    \right\vert\kern-0.25ex\right\vert\kern-0.25ex\right\vert}}

\geometry{left=3cm,right=3cm,top=2.5cm,bottom=2cm}

\newcommand{\tdiv}{\mathrm{div}}
\newcommand{\avg}[1]{\ensuremath{\left\{\!\!\left\{ #1\right\}\!\!\right\}}}
\newcommand{\jump}[1]{\ensuremath{\left\llbracket #1\right\rrbracket}}

\begin{document}           % End of preamble and beginning of text.

\title{A Robin-type domain decomposition method for a novel mixed-type DG method for the coupled Stokes-Darcy problem}
\author{Lina Zhao\footnotemark[1]}
\renewcommand{\thefootnote}{\fnsymbol{footnote}}
\footnotetext[1]{Department of Mathematics, City University
of Hong Kong, Kowloon Tong, Hong Kong SAR, China. ({linazha@cityu.edu.hk}). This author was supported by a grant from City University of Hong Kong (Project No. 7200699).}
%\footnotetext[2]{Department of Mathematics, The Chinese University
%of Hong Kong, Hong Kong SAR, China. ({tschung@math.cuhk.edu.hk}).}
\maketitle

\begin{abstract}
In this paper, we first propose and analyze a novel mixed-type DG method for the coupled Stokes-Darcy problem on simplicial meshes. The proposed formulation is locally conservative. A mixed-type DG method in conjunction with the stress-velocity formulation is employed for the Stokes equations, where the symmetry of stress is strongly imposed. The staggered DG method is exploited to discretize the Darcy equations. As such, the discrete formulation can be easily adapted to account for the Beavers-Joseph-Saffman interface conditions without introducing additional variables. Importantly, the continuity of normal velocity is satisfied exactly at the discrete level. A rigorous convergence analysis is performed for all the variables. Then we devise and analyze a domain decomposition method via the use of Robin-type interface boundary conditions, which allows us to solve the Stokes subproblem and the Darcy subproblem sequentially with low computational costs. The convergence of the proposed iterative method is analyzed rigorously. In particular, the proposed iterative method also works for very small viscosity coefficient. Finally, several numerical experiments are carried out to demonstrate the capabilities and accuracy of the novel mixed-type scheme, and the convergence of the domain decomposition method.

\end{abstract}


\pagestyle{myheadings} \thispagestyle{plain}
\markboth{ZhaoChung} {DDM for the coupled Stokes-Darcy problem}

% importance of the model problem, the deficiencies of existing works (to address the advantages of our scheme), our novel contributions

\section{Introduction}

%We use the velocity-pressure-stress formulation for the Stokes equations as have done in \cite{Bochev95}, we also use mixed discontinuous Galerkin method for Darcy equations (cf. \cite{Wen20}). Choose proper integration by parts and proper polynomial orders. The interface condition can be imposed via integration by parts and Nitsche technique.
%
%
%\Red{velocity-pressure-stress formulation, stress is strongly symmetric, locally conservative? allow unfitted meshes. polyDG. derive a new formulation}

Coupling incompressible flow and porous media flow has drawn great attention over the past years, which has been involved in many practical applications, such as ground water contamination and industrial filtration. This coupled phenomenon can be mathematically expressed by the Stokes-Darcy problem, where the free fluid region is governed by the Stokes equations and the porous media region is described by Darcy's law, and three transmission conditions are prescribed on the interface (cf. \cite{Beavers67,Saffman71}).

The devising of numerical schemes for the coupled Stokes-Darcy problem hinges on a suitable choice of stable pairs for both the Stokes equations and Darcy equations. As it is well known, the standard mixed formulations for the Stokes equations and Darcy equations earn different compatibility conditions, thus a straightforward application of the existing solvers for the Stokes equations and Darcy equations may not be feasible.
To this end, a great amount of effort has been devoted to developing accurate and efficient numerical schemes for the coupled Stokes-Darcy problem, and a non-exhaustive list of these approaches include Lagrange multiplier methods \cite{Layton02,Gatica08,Vassilev09,Gatica11}, weak Galerkin method \cite{Chen15,Li18}, strongly conservative methods \cite{Kanschat10,Fu18}, stabilized mixed finite element method \cite{Rui17,Mahbub19}, discontinuous Galerkin (DG) methods \cite{Eiviere05,Wen20}, virtual element method \cite{Liu19,Wang19}, a lowest-order staggered DG method \cite{Zhao20} and penalty methods \cite{Zhou21}. The coupled Stokes-Darcy problem describes multiphysics phenomena, and involves a Stokes subproblem and a Darcy subproblem, it is thus natural to resort to domain decomposition methods, which allows one to solve the coupled system sequentially with a low computational cost. Various domain decomposition methods have been developed for the coupled Stokes-Darcy problem, see, for example \cite{Discacciati07,Cao11,Chendd11,Cao14,Vassilev14,He15}, most of which are based on velocity-pressure formulation of the Stokes equations.

%One can refer to some of the existing works \cite{Discacciati07,Cao14,Vassilev14,He15} in the context of domain decomposition methods for the coupled Stokes-Darcy problem.

The first purpose of this paper is to devise and analyze a novel mixed-type method for the coupled Stokes-Darcy problem. In the proposed formulation, we use a mixed-type DG method in conjunction with stress-velocity formulation for the discretization of the Stokes subproblem, and we use staggered DG method introduced in \cite{Chung09} for the discretization of the Darcy subproblem. Unlike the schemes proposed in \cite{Riviere05,Wen20}, we enforce the normal continuity of velocity directly into the formulation of the method without resorting to Lagrange multiplier. This is based on our observation that the degrees of freedom for the Darcy velocity space is defined by using dual edge degrees of freedom and interior degrees of freedom, and the interface can be treated as the union of the primal edges. Therefore, the normal continuity of velocity can be simply imposed into the discrete formulation by replacing the Darcy's normal velocity by the Stokes' normal velocity, which in conjunction with a suitable decomposition of the stress variable on the interface yields the resulting formulation. It is worth mentioning that the normal continuity of velocity is satisfied exactly at the discrete level.
The advantages of the proposed formulation are multifold: First, the symmetry of stress is strongly imposed. The stress variable has a physical meaning and its computation is also important from a practical point of view. Second, pressure is eliminated from the equation, which reduces the size of the global system. Third, the continuity of normal velocity can be imposed directly without using Lagrange multiplier. A rigorous convergence analysis is carried out for all the variables. The proof for the optimal convergence of $L^2$-error of Darcy velocity is non-trivial, to overcome this issue, we exploit a non-standard trace theorem. We remark that the proposed scheme utilizes stress-velocity formulation for the Stokes equations, which is rarely explored for the coupled Stokes-Darcy problem in the existing literature. Our work will shed new insights into the devising and analysis of novel numerical schemes for the coupled Stokes-Darcy problem. We also address that the numerical results demonstrate that the proposed scheme also works well for small values of viscosity.

The proposed global formulation involves four variables: stress, fluid velocity, Darcy velocity and Darcy pressure, which may require high computational costs especially for large scale problems. To reduce the computational costs, we aim to devise a domain decomposition method based on the proposed spatial discretization, where the global formulation is decomposed into the Stokes subproblem and the Darcy subproblem by using newly constructed Robin-type interface conditions. The construction of the novel Robin-type interface condition is not a simple extension of the method introduced in \cite{Sun21}, instead it takes advantage of the special features of staggered DG method. Moreover, the compatibility conditions are derived to ensure the equivalence of the modified discrete formulation and the original discrete formulation. The convergence of the Robin-type domain decomposition method is rigorously analyzed with the help of the compatibility conditions. Our convergence analysis indicates that the convergence of the Robin-type domain decomposition method is $1-\mathcal{O}(h)$ for $\delta_p=\delta_f$, where $\delta_p$ and $\delta_f$ are parameters introduced in Section~\ref{sec:robin}. Moreover, when $\delta_p$ and $\delta_f$ satisfy certain conditions, then the convergence rate of the Robin-type domain decomposition method is $h$-independent, which is particularly inspiring. Several numerical experiments are carried out to demonstrate the performance of the proposed scheme. We can observe that the proposed domain decomposition method converges as reflected by the theories.
In particular, it also converges for small values of viscosity.

The rest of the paper is organized as follows. In the next section, we describe the model problem and derive the stress-velocity formulation for the Stokes equations. In Section~\ref{sec:scheme}, we derive the discrete formulation and prove the unique solvability. The convergence error estimates for all the variables measured in $L^2$-error are given in Section~\ref{sec:priori}. Then, the Robin-type domain decomposition method is constructed and analyzed in Section~\ref{sec:robin}. Several numerical experiments are carried out in Section~\ref{sec:numerical} to verify the proposed theories. Finally, a concluding remark is given in Section~\ref{sec:conclusion}.

\section{Model problem}

Let $\Omega:=\Omega_S\cup \Omega_D$ denote the polygonal domain in $\mathbb{R}^2$, where $\Omega_S$ and $\Omega_D$ represent the fluid domain and porous media domain, respectively. Let $\Gamma$ denote the interface between $\Omega_S$ and $\Omega_D$, and let $\Gamma_i = \partial \Omega_i\backslash \Gamma$ for $i\in \{S,D\}$, see Figure~\ref{fig:domain} for an illustration of the computational domain. We use $\bm{n}_i (i=S,D)$ to represent the unit normal vector to $\partial \Omega_i$. Let $\bm{t}_S$ be an orthonormal system of tangential vectors on $\Gamma$.

\begin{figure}[t]
\centering
\includegraphics[width=0.4\textwidth]{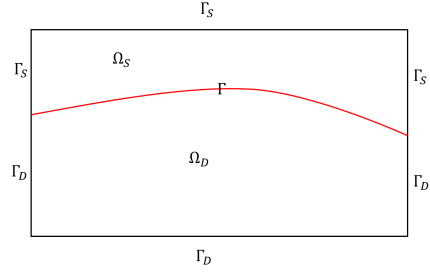}
\caption{Profile of the computational domain.}
\label{fig:domain}
\end{figure}

In $\Omega_S$, the fluid flow is governed by the Stokes equations
\begin{alignat}{2}
-\mbox{div}\,\underline{\sigma}&= \bm{f}_S&&\quad \mbox{in}\;\Omega_S,\label{eq:stokes1}\\
\underline{\sigma}&=2\mu\varepsilon(\bm{u}_S)-pI&&\quad \mbox{in}\;\Omega_S,\label{eq:sigmap}\\
\text{div}\,\bm{u}_S&=0&&\quad \mbox{in}\;\Omega_S,\label{eq:stokes2}\\
\bm{u}_S&=\bm{0}&&\quad \mbox{on}\;\Gamma_S,\label{eq:stokes3}
\end{alignat}
where $\mu$ is the viscosity coefficient which is assumed to be a positive constant, $\bm{u}_S$ is the fluid velocity, $p$ is the fluid pressure, $\underline{\sigma}$ is the stress tensor, $I$ is the $2\times 2$ identity matrix and $\varepsilon(\bm{u}_S)$ is the deformation tensor defined by $\varepsilon(\bm{u}_S)=\frac{\nabla \bm{u}_S+\nabla \bm{u}_S'}{2}$. Hereafter, we use $A'$ to represent the transpose of $A$.
%which can be rewritten as
%\begin{alignat}{2}
%-\tdiv \bm{\sigma}&=\bm{f}&&\quad \mbox{in}\;\Omega,\label{eq:first1}\\
%\mathcal{A}\bm{\sigma}&=2\nu\varepsilon(\bm{u}_s)&&\quad \mbox{in}\;\Omega\label{eq:first2},\\
%\bm{u}_s&=\bm{0}&&\quad \mbox{on}\;\partial\Omega.
%\end{alignat}
%where
%\begin{align*}
%\mathcal{A}\bm{\sigma}=\bm{\sigma}-\frac{1}{2}\mbox{tr}(\bm{\sigma})I.
%\end{align*}

In $\Omega_D$, the governing equations are Darcy equations
\begin{align}
\tdiv \bm{u}_D&=f_D\;\mbox{in}\; \Omega_D,\label{eq:Darcy1}\\
\bm{u}_D+K \nabla p_D&=0\quad \mbox{in}\; \Omega_D\label{eq:Darcy2},\\
p_D&=0\quad \mbox{on}\; \Gamma_D.
\end{align}
%Interface conditions
%\begin{alignat}{2}
%\bm{u}_s\cdot\bm{n}_{12}&=\bm{u}_D\cdot\bm{n}_{12}&&\quad \mbox{in}\; \Gamma_{12},\label{eq:interface1}\\
%p_S-2\mu \varepsilon(\bm{u}_S)\bm{n}_{12}\cdot\bm{n}_{12}&=p_D&&\quad \mbox{in}\; \Gamma_{12},\label{eq:interface2}\\
%\bm{u}_1\cdot\bm{t}_{12}&=-2G\mu \varepsilon(\bm{u}_S)\bm{n}_{12}\cdot\bm{t}_{12}&&\quad \mbox{in}\; \Gamma_{12}.\label{eq:interface3}
%\end{alignat}
On the interface $\Gamma$, we prescribe the following Beavers-Joseph-Saffman conditions (cf. \cite{Beavers67,Saffman71})
\begin{alignat}{2}
\bm{u}_S\cdot\bm{n}_{S}&=\bm{u}_D\cdot\bm{n}_{S}&&\quad \mbox{on}\; \Gamma,\label{eq:interface1}\\
-\underline{\sigma}\bm{n}_{S}\cdot \bm{n}_{S}&=p_D&&\quad \mbox{on}\; \Gamma\label{eq:interface22},\\
\bm{u}_S\cdot\bm{t}_{S}&=-G \underline{\sigma}\bm{n}_{S}\cdot\bm{t}_{S}&&\quad \mbox{on}\; \Gamma,\label{eq:interface32}
\end{alignat}
where $G>0$ is the phenomenological friction coefficient.

Now we will derive a stress-velocity formulation based on \eqref{eq:stokes1}-\eqref{eq:stokes2} by eliminating $p$. First, we have
\begin{align*}
\mbox{tr}(\varepsilon(\bm{u}_S))=\mbox{div}\bm{u}_S=0,\\
\mbox{tr}(\underline{\sigma})=2\mu\mbox{tr}(\varepsilon(\bm{u}_S))-2p=-2p,
\end{align*}
thus, we obtain
\begin{align*}
p=-\frac{1}{2}\mbox{tr}(\underline{\sigma}).
\end{align*}
Consequently, we can recast \eqref{eq:sigmap} into the following equivalent formulation
\begin{align*}
\mathcal{A}\underline{\sigma}=2\mu\varepsilon(\bm{u}_S),
\end{align*}
where
\begin{align*}
\mathcal{A}\underline{\sigma}=\underline{\sigma}-\frac{1}{2}\mbox{tr}(\underline{\sigma})I.
\end{align*}
Note that  $\mathcal{A}\underline{\sigma}$ is a trace-free tensor called
deviatoric part and $\text{ker}(\mathcal{A})=\{q I\;| \;q\; \text{is a scalar function}\}$.

Then we can rewrite \eqref{eq:stokes1}-\eqref{eq:stokes2} as the following equivalent system:
\begin{alignat}{2}
-\mbox{div}\,\underline{\sigma}&=\bm{f}&&\quad \mbox{in}\;\Omega_S,\label{eq:first1}\\
\mathcal{A}\underline{\sigma}&=2\mu\varepsilon(\bm{u}_S)&&\quad \mbox{in}\;\Omega_S\label{eq:first2},\\
\bm{u}_S&=\bm{0}&&\quad \mbox{on}\;\Gamma_S.
\end{alignat}

We introduce some notations that will be used throughout the paper. Let $D\subset \mathbb{R}^d,d=1,2$. By $(\cdot,\cdot)_D$, we denote the scalar product in $L^2(D):(p,q)_D:=\int_D p \,q\,dx$. We use the same notation for the scalar product in $L^2(D)^2$ and in $L^2(D)^{2\times 2}$. More precisely, $(\bm{\xi},\bm{w})_D:=\sum_{i=1}^2 (\xi^i,w^i)$ for $\bm{\xi},\bm{w}\in L^2(D)^2$ and $(\underline{\psi},\underline{\zeta})_D:=\sum_{i=1}^2\sum_{j=1}^2 (\psi^{i,j},\zeta^{i,j})_D$ for $\underline{\psi},\underline{\zeta}\in L^2(D)^{2\times 2}$. The associated norm is denoted by $\|\cdot\|_{0,D}$.
Given an integer $m>0$, we denote the scalar-valued Sobolev spaces by $H^m(D) =W^{m,2}(D)$ with the norm $\|\cdot\|_{m,D}$ and seminorm $|\cdot|_{m,D}$. In addition, we use $H^m(D)^d$ and $H^m(D)^{d\times d}$ to denote the vector-valued and tensor-valued Sobolev spaces, respectively. In the following, we use $C$ to denote a generic constant independent of the meshsize which may have different values at different occurrences.

\section{The new scheme}\label{sec:scheme}

In this section, we are devoted to the derivation of a novel mixed-type DG method for the coupled Stokes-Darcy system \eqref{eq:Darcy1}-\eqref{eq:interface32} and \eqref{eq:first1}-\eqref{eq:first2}. To this end, we first introduce the following meshes.
Following \cite{Chung09,Zhao18,Zhao21}, we first let $\mathcal{T}_{u,D}$ be the
initial partition of the domain $\Omega_D$ into non-overlapping triangular meshes as shown in the left panel of Figure~\ref{fig:grid}. We use $\mathcal{F}_{h,\Gamma}$ to denote the set of edges lying on the interface $\Gamma$.
We also let $\mathcal{F}_{pr,D}$ be the set of all edges excluding the interface edges in the initial partition $\mathcal{T}_{u,D}$ and $\mathcal{F}_{pr,D}^{0}\subset \mathcal{F}_{pr,D}$ be the
subset of all interior edges of $\Omega_D$.
For each triangular mesh $E$ in the initial partition $\mathcal{T}_{u,D}$, we construct the sub-triangulation by connecting an interior point $\nu$ to
all the vertices of $E$. For simplicity, we select $\nu$ as the center point.
We rename the union of these triangles sharing the common point $\nu$ by $S(\nu)$. %We remark that $S(\nu)$ is the rectangular mesh in the initial partition.
Moreover, we will use $\mathcal{F}_{dl,D}$ to denote the set of all the new edges generated by this subdivision process and
use $\mathcal{T}_{h,D}$ to denote the resulting quasi-uniform triangulation,
on which our basis functions are defined. Here the sub-triangulation $\mathcal{T}_{h,D}$ satisfies standard mesh regularity assumptions (cf. \cite{Ciarlet78}).
% Furthermore,
%$\mathcal{T}_{h_i}$ is assumed to satisfy local quasi-uniform assumption in the sense that for any pair of elements $\tau$ and
%$\tau'$ in $\mathcal{T}_{h_i}$ which share an edge, there exists a
%constant $\kappa$ independent of $h_\tau$ and $h_{\tau'}$ such
%that $\kappa^{-1} \le h_\tau/h_{\tau'} \le \kappa$.
%In addition, we define $\mathcal{F}_D:=\mathcal{F}_{pr,D}\cup \mathcal{F}_{dl,D}$, $\mathcal{F}_D^{0}:=\mathcal{F}_{pr,D}^{0}\cup \mathcal{F}_{dl,D}$.
% Also, we let $h_e$ denote the length of edge $e\in \mathcal{F}_{h,D}$.
This construction is illustrated in Figure~\ref{fig:grid},
where the black solid lines are edges in $\mathcal{F}_{pr,D}^0$
and the red dotted lines are edges in $\mathcal{F}_{dl,D}$. For each interior edge $e\in \mathcal{F}_{pr,D}^0$, we use $D(e)$ to denote the union of the two triangles in $\mathcal{T}_{h,D}$ sharing the edge $e$,
and for each boundary edge $e\in(\mathcal{F}_{pr,D}\cup \mathcal{F}_{h,\Gamma})\backslash\mathcal{F}_{pr,D}^0$, we use $D(e)$ to denote the triangle in $\mathcal{T}_{h,D}$ having the edge $e$,
see Figure~\ref{fig:grid}.

On the other hand, we let $\{\mathcal{T}_{h,S}\}$ be a family of shape-regular triangulations of $\bar{\Omega}_S$. For simplicity, we assume that the meshes between the two domains $\Omega_S$ and $\Omega_D$ match at the interface. We use $\mathcal{F}_{h,S}$ to denote the set of all edges of $\mathcal{T}_{h,S}$ excluding the edges lying on the interface. We use $\mathcal{F}_{h,S}^0$ to represent the subset of $\mathcal{F}_{h,S}$, i.e., $\mathcal{F}_{h,S}^0$ is the union of interior edges of $\bar{\Omega}_S$. In addition, we let $\mathcal{T}_h=\mathcal{T}_{h,S}\cup \mathcal{T}_{h,D}$ and $\mathcal{F}_h=\mathcal{F}_{h,S}\cup \mathcal{F}_{h,D}$. In what follows, $h_e$ stands for the length of edge $e\in \mathcal{F}_{h}$. For each triangle
$T\in \mathcal{T}_{h}$, we let $h_T$ be the diameter of $T$
%$h_S=\max_{T\in \mathcal{T}_{h,S}}h_T$,
%$h_D=\max_{T\in \mathcal{T}_{h,D}}h_T$,
and
$h=\max_{T\in \mathcal{T}_{h}}h_T$. For each interior edge $e$, we then fix $\bm{n}_{e}$ as one of the two possible unit normal vectors on $e$.
When there is no ambiguity,
we use $\bm{n}$ instead of $\bm{n}_{e}$ to simplify the notation. In addition, we use $\bm{t}$ to represent the corresponding unit tangent vector.
To simplify the presentation, we only consider triangular meshes in this paper, and the extension to polygonal meshes will be investigated in the future paper.

For $k\geq 1$, $T\in \mathcal{T}_{h}$ and $e\in \mathcal{F}_h$, we define $P^k(T)$ and $P^k(e)$ as the spaces of polynomials of degree up to order $k$ on $T$ and $e$, respectively. For a scalar or vector function $v$ belonging to the broken Sobolev space, its jump and average on an interior edge $e$ are defined as
\begin{equation*}
\jump{v}_e=v_{1}-v_{2},\quad \avg{v}_e=\frac{v_{1}+v_{2}}{2},
\end{equation*}
where $v_j=v_{T_j},j=1,2$ and $T_{1}$, $T_{2}$ are the
two triangles in $\mathcal{T}_{h}$ having the edge $e$. For the boundary edges, we simply define $\jump{v}_e=v_{1}$ and $\avg{v}_e=v_{1}$. We can omit the subscript $e$ when it is clear which edge we are referring to.

\begin{figure}[t]
\centering
\includegraphics[width=0.35\textwidth]{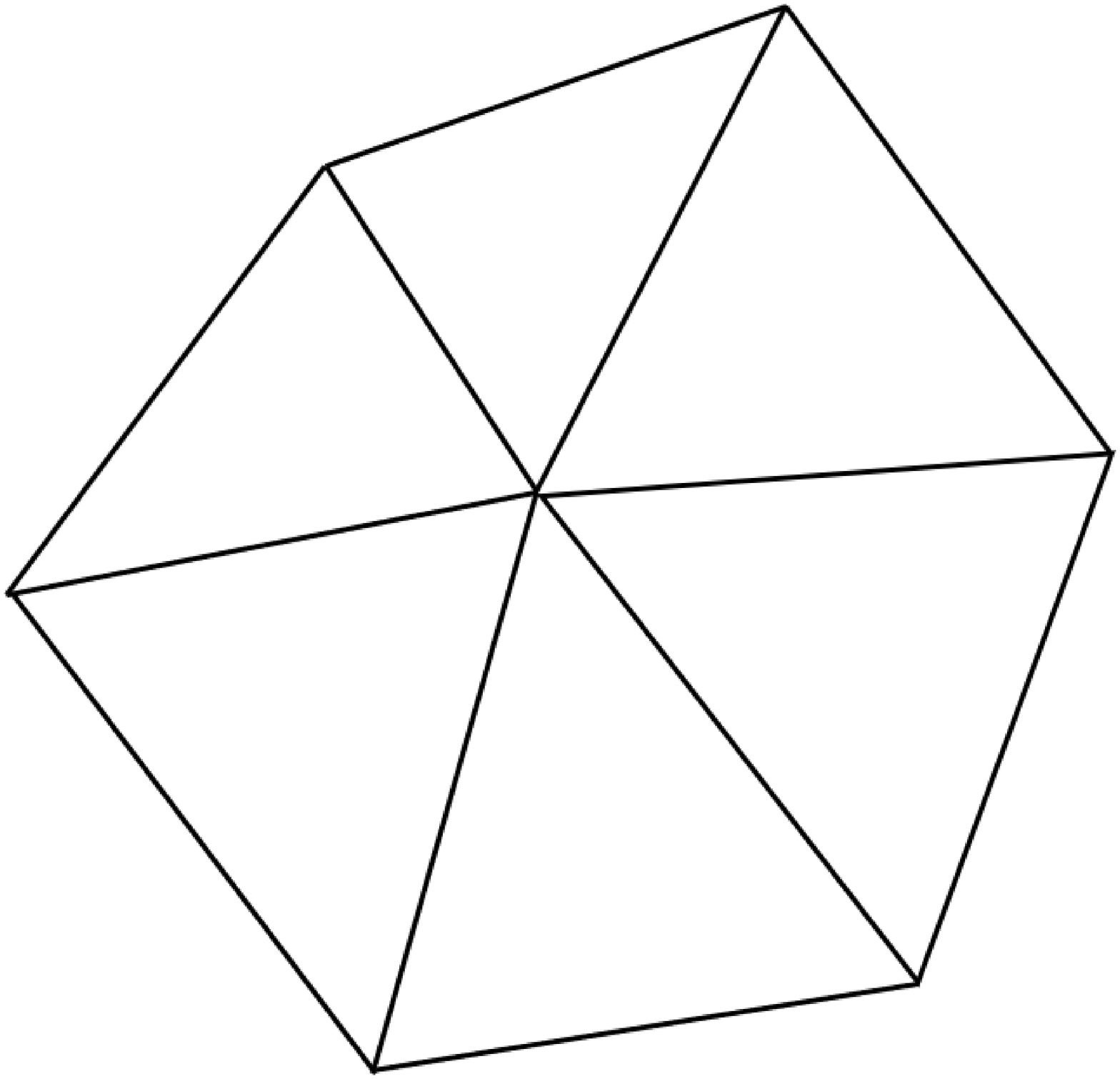}
\includegraphics[width=0.35\textwidth]{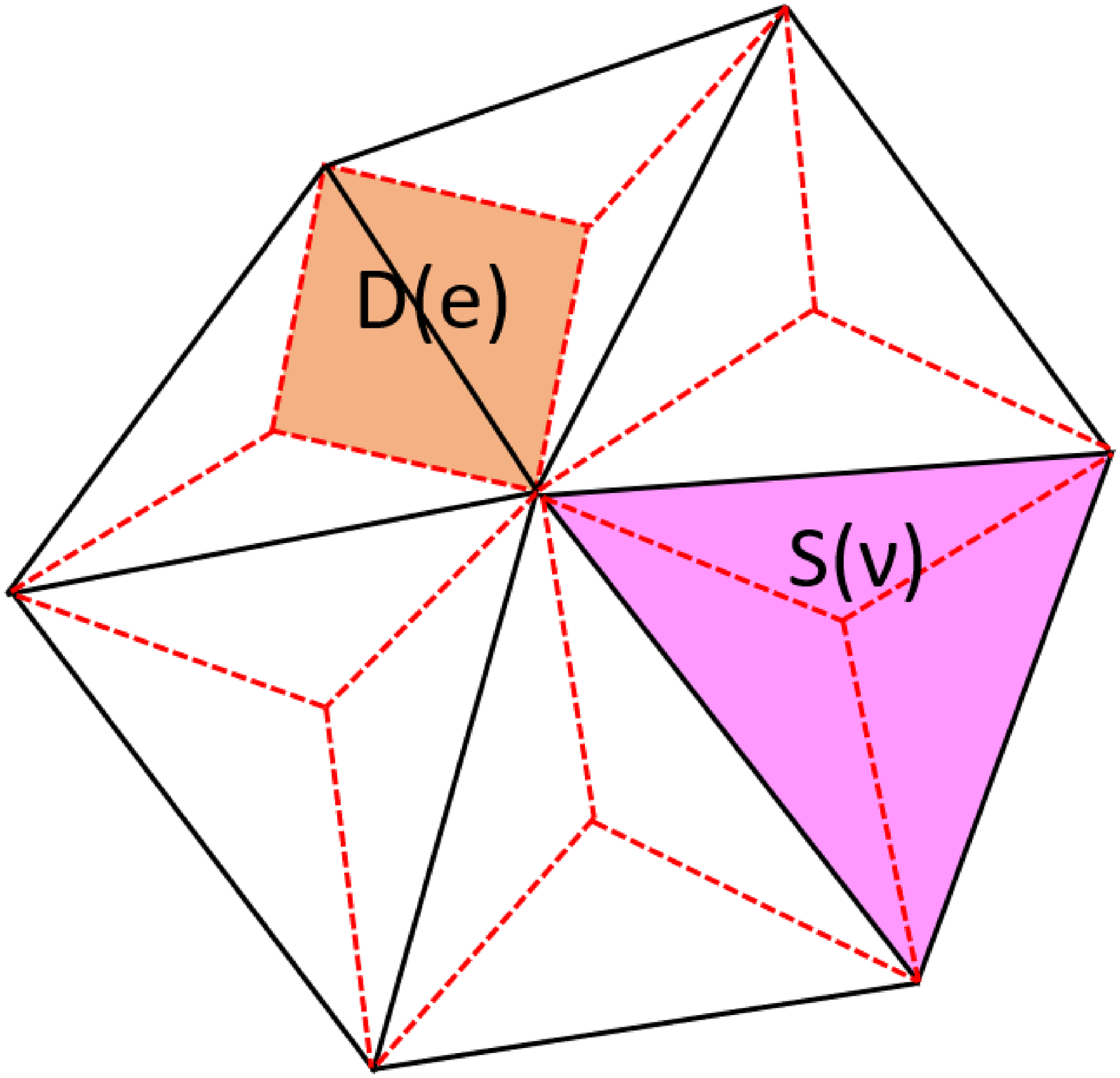}
\caption{The schematics of the meshes for $\Omega_D$. Primal meshes (left), dual meshes and simplicial meshes (right). The solid lines represent the primal edges and the dashed lines represent the dual edges. $S(\nu)$ represents the primal element and $D(e)$ represents the dual element.}
\label{fig:grid}
\end{figure}

We define the following spaces for the numerical approximations.
\begin{align*}
\Sigma_{h}^S:&=\{\underline{w}:\underline{w}=\underline{w}',\underline{w}|_T\in P_{k-1}(T)^{2\times 2}, \;\forall T\in \mathcal{T}_{h,S}\},\\
U_{h}^S:&=\{\bm{v}_S:\bm{v}_S|_T\in P_{k}(T)^{2}, \;\forall T\in \mathcal{T}_{h,S};\bm{v}_S|_{\Gamma_S}=\bm{0}\}
\end{align*}
and
\begin{align*}
U_{h}^D:&=\{\bm{v}_D:\bm{v}_D|_T\in P_{k}(T)^{2}, \;\forall T\in \mathcal{T}_{h,D},\jump{\bm{v}_D\cdot\bm{n}}_e=0\;\forall e\in \mathcal{F}_{dl,D}\},\\
P_h^D:&=\{q:q|_T\in P_{k}(T), \;\forall T\in \mathcal{T}_{h,D}, \jump{q}_e=0\;\forall e\in \mathcal{F}_{pr,D}^0; q|_{\Gamma_D}=0\}.
\end{align*}
In the following, we use $\varepsilon_h(\bm{v})$ to represent the element-wise defined
deformation tensor, i.e., $\varepsilon_h(\bm{v})\mid_T=\frac{\nabla \bm{v}\mid_T+\nabla \bm{v}'\mid_T}{2}$ for $\bm{v}\in U_h^S$. Hereafter, we use $\nabla_h$ and $\tdiv_h$ to represent the gradient and divergence operators defined element by element, respectively.

For later use, we specify the degrees of freedom for $P_h^D$ as follows.
    \begin{itemize}
        \item[(SD1)]
            For $e\in \mathcal{F}_{pr,D}^0\cup \mathcal{F}_{h,\Gamma}$, we have
            \begin{equation}
                \phi_e(q) :=( q,p_k)_e\quad \forall p_k\in P_k(e).\label{eq:dof1}
            \end{equation}
        \item[(SD2)]
            For each $T\in \mathcal{T}_{h,D}$, we define
            \begin{equation}
                \phi_T(q) := (q,p_{k-1})_T\quad \forall p_{k-1}\in P_{k-1}(T).\label{eq:dof2}
            \end{equation}
    \end{itemize}
%\begin{alignat}{2}
%((2\nu)^{-1}\mathcal{A}\bm{\sigma}_h,\bm{w}) &= B_h(\bm{w},\bm{u}_{h,s})\quad&& \forall \bm{w}\in \Sigma_h,\label{eq:discrete1}\\
%B_h(\bm{\sigma}_h,\bm{v})+\sum_{e\in \mathcal{F}_h^0}\tau h_e^{-1}(\jump{\bm{u}_{h,s}},\jump{\bm{v}})_e&=(\bm{f},\bm{v})\quad &&\forall \bm{v}\in U_h\label{eq:discrete2},
%\end{alignat}
%%
%where
%\begin{align*}
%B_h(\bm{\sigma}_{h}, \bm{v})&=-\sum_{e\in \mathcal{F}_{h}^0}(\avg{\bm{\sigma}_{h} \bm{n}}, \jump{\bm{v}})_{e}+\sum_{T\in \mathcal{T}_{h}}(\bm{\sigma}_{h},\varepsilon_h(\bm{v}))_T.
%\end{align*}
%
%
%Integration by parts implies that
%\begin{align}
%B_h(\underline{\sigma}_{h}, \bm{v})=\sum_{e\in \mathcal{F}_{h}^0}(\jump{\underline{\sigma}_{h} \bm{n}}, \avg{\bm{v}})_{e}-\sum_{T\in \mathcal{T}_{h}}(\mbox{div}_h\, \underline{\sigma}_h,\bm{v})_T.\label{eq:adjoint}
%\end{align}
%
%
%\begin{align*}
%(\tdiv \bm{\sigma},\bm{w})=\sum_{e\in \mathcal{F}_{h}^0}(\avg{\bm{\sigma}\bm{n}},\jump{\bm{w}})_e+\sum_{e\in \Gamma}(\bm{\sigma}\bm{n},\bm{w})_e-(\bm{\sigma},\nabla_h_h \bm{w})
%\end{align*}
%
%
%\begin{align*}
%\sum_{e\in \Gamma}(\bm{\sigma}\bm{n},\bm{w})_e&=\sum_{e\in \Gamma}(\bm{\sigma}\bm{n}\cdot\bm{n},\bm{w}\cdot\bm{n})_e+\sum_{e\in \Gamma}(\bm{\sigma}\bm{n}\cdot\bm{t},\bm{w}\cdot\bm{t})_e\\
%&=-\sum_{e\in \Gamma}(p,\bm{w}\cdot\bm{n})_e-\sum_{e\in \Gamma}(\bm{u}_s\cdot\bm{t},\bm{w}\cdot\bm{t})_e
%\end{align*}
%
%
%\begin{align*}
%(\tdiv \bm{u}_D,q)=\sum_{e\in \mathcal{F}_{h,D}}(\avg{\bm{u}_D\cdot\bm{n}},\jump{q})_e+\sum_{e\in \Gamma}(\bm{u}_S\cdot\bm{n},q)_e-(\bm{u}_D,\nabla_h q)
%\end{align*}
%
Now we are ready to derive the discrete formulation based on the first order system \eqref{eq:Darcy1}-\eqref{eq:Darcy2} and \eqref{eq:first1}-\eqref{eq:first2}. Multiplying \eqref{eq:stokes1} by a test function $\bm{v}_S\in U_h^S$ and applying integration by parts yield
\begin{align*}
-(\tdiv \bm{\sigma},\bm{v}_S)_{\Omega_S}&=-\sum_{e\in \mathcal{F}_{h,S}^0}(\avg{\bm{\sigma}\bm{n}},\jump{\bm{v}_S})_e-\sum_{e\in \mathcal{F}_{h,\Gamma}}(\bm{\sigma}\bm{n}_S,\bm{v}_S)_e+(\bm{\sigma}, \varepsilon_h(\bm{v}_S))_{\Omega_S}.
\end{align*}
where the second term can be recast into the following form by using \eqref{eq:interface22} and \eqref{eq:interface32}
\begin{align*}
-\sum_{e\in \mathcal{F}_{h,\Gamma}}(\bm{\sigma}\bm{n}_S,\bm{v}_S)_e&=-\sum_{e\in \mathcal{F}_{h,\Gamma}}(\bm{\sigma}\bm{n}_S,(\bm{v}_S\cdot\bm{n}_S)\bm{n}_S+(\bm{v}_S\cdot\bm{t}_S)\bm{t}_S)_e\\
&=\sum_{e\in \mathcal{F}_{h,\Gamma}}(p_D,\bm{v}_S\cdot\bm{n}_S)_e+\frac{1}{G}\sum_{e\in \mathcal{F}_{h,\Gamma}}(\bm{u}_S\cdot\bm{t}_S,\bm{v}_S\cdot\bm{t}_S)_e.
\end{align*}
Multiplying \eqref{eq:Darcy2} by a test function $\bm{v}_D\in U_h^D$ and performing integration by parts lead to
\begin{align*}
(\nabla p_D,\bm{v}_D)_{\Omega_D}=\sum_{e\in \mathcal{F}_{h,\Gamma}}( \bm{v}_D\cdot\bm{n}_D,p_D)_e+\sum_{e\in \mathcal{F}_{pr,D}^0}(p_D,\jump{\bm{v}_D\cdot\bm{n}})_e-(p_D,\tdiv_h \bm{v}_D)_{\Omega_D}.
\end{align*}
Finally we multiply \eqref{eq:Darcy1} by $q_D\in P_h^D$, then using the interface condition \eqref{eq:interface1} implies that
\begin{align*}
(\tdiv \bm{u}_D,q_D)_{\Omega_D}&=\sum_{e\in \mathcal{F}_{dl,D}}(\bm{u}_D\cdot\bm{n},\jump{q_D})_e+\sum_{e\in \mathcal{F}_{h,\Gamma}}(\bm{u}_D\cdot\bm{n}_D,q_D)_e-(\bm{u}_D,\nabla_h q_D)_{\Omega_D}\\
&=\sum_{e\in \mathcal{F}_{dl,D}}(\bm{u}_D\cdot\bm{n},\jump{q_D})_e-(\bm{u}_D,\nabla_h q_D)_{\Omega_D}
-\sum_{e\in \mathcal{F}_{h,\Gamma}}(\bm{u}_S\cdot\bm{n}_S,q_D)_e.
\end{align*}
Based on the above derivations, we are now in position to define the following bilinear forms, which is instrumental for later use.
\begin{align*}
a_{S}(\bm{\sigma}_h,\bm{v}_S)&=-\sum_{e\in \mathcal{F}_{h,S}^0}(\avg{\bm{\sigma}_h\bm{n}},\jump{\bm{v}_S})_e+(\bm{\sigma}_h, \varepsilon_h( \bm{v}_S))_{\Omega_S}
%b_{1,S}^*(p_S,\bm{v}_S)&=-\sum_{e\in \mathcal{F}_{h,S}}(\avg{p_S},\jump{\bm{v}_S\cdot\bm{n}})_e+(p_S,\nabla_h \cdot\bm{v}_S)_{\Omega_S},\\
%b_{1,S}(\bm{u}_S,q_S)&=\sum_{e\in \mathcal{F}_{h,S}}(\avg{\bm{u}_S\cdot\bm{n}},\jump{q_S})_e+(\bm{u}_S\cdot\bm{n}_S,q_S)_\Gamma-(\bm{u}_S,\nabla_h q_S)_{\Omega_S}
\end{align*}
and
\begin{align*}
b_{D}^*(p_D,\bm{v}_D)&=-\sum_{e\in \mathcal{F}_{pr,D}^0}(p_D,\jump{\bm{v}_D\cdot\bm{n}})_e+(p_D,\tdiv_h \bm{v}_D)_{\Omega_D}-\sum_{e\in \mathcal{F}_{h,\Gamma}}( \bm{v}_D\cdot\bm{n}_D,p_{D})_e,\\
b_{D}(\bm{u}_D,q_D)&=\sum_{e\in \mathcal{F}_{dl,D}}(\bm{u}_D\cdot\bm{n},\jump{q_D})_e-(\bm{u}_D,\nabla_h q_D)_{\Omega_D}.
\end{align*}
It follows from integration by parts that
\begin{align*}
a_S(\bm{\sigma}_h,\bm{v}_S)=\sum_{e\in \mathcal{F}_{h,S}^0}(\jump{\bm{\sigma}_h\bm{n}}, \avg{\bm{v}_S})_e-(\tdiv_h\underline{\sigma}_h,\bm{v}_S)_{\Omega_S}\quad \forall \bm{v}_S\in U_h^S.
\end{align*}
The discrete formulation reads as follows: Find $(\underline{\sigma}_h,\bm{u}_{S,h})\in \Sigma_h^S\times U_h^S$ and $(\bm{u}_{D,h},p_{D,h})\in U_h^D\times P_h^D$ such that
\begin{align}
&((2\mu)^{-1}\mathcal{A}\underline{\sigma}_h,\underline{w})_{\Omega_S} = a_S(\underline{w},\bm{u}_{S,h})\quad \forall \underline{w}\in \Sigma_h^S,\label{eq:discrete1}\\
&a_S(\underline{\sigma}_h,\bm{v}_S)+(K^{-1}\bm{u}_{D,h},\bm{v}_D)_{\Omega_D}-b_{D}^*(p_{D,h},\bm{v}_D)+\sum_{e\in \mathcal{F}_{h,\Gamma}}( \bm{v}_S\cdot\bm{n}_S,p_{D,h})_e\nonumber\\
&\;+\sum_{e\in \mathcal{F}_{h,S}^0}\gamma h_e^{-1}(\jump{\bm{u}_{S,h}},\jump{\bm{v}_S})_e+\frac{1}{G}\sum_{e\in \mathcal{F}_{h,\Gamma}}(\bm{u}_{S,h}\cdot\bm{t}_S,\bm{v}_S\cdot\bm{t}_S)_e=(\bm{f}_S,\bm{v}_S)_{\Omega_S}\quad \forall \bm{v}_S\in U_h^S,\label{eq:discrete2}\\
&b_{D}(\bm{u}_{D,h},q_D)-\sum_{e\in \mathcal{F}_{h,\Gamma}}(\bm{u}_{S,h}\cdot\bm{n}_S,q_D)_e=(f_D,q_D)_{\Omega_D}\quad \forall q_D\in P_h^D,\label{eq:discrete3}
\end{align}
where $\gamma>0$ is a constant over each edge $e\in \mathcal{F}_{h,S}^0$.

Integration by parts implies the following discrete adjoint property
\begin{align}
b_{D}(\bm{v}_D,q_D)&=b_{D}^*(q_D, \bm{v}_D)\quad \forall (\bm{v}_D,q_D)\in U_h^D\times P_h^D\label{eq:adjointD2}.
\end{align}

To facilitate later analysis, we define the following mesh-dependent norm/semi-norm
\begin{align*}
\|\bm{v}_D\|_{0,h}^2&=\|\bm{v}_D\|_{0,\Omega_D}^2+\sum_{e\in \mathcal{F}_{dl,D}}h_e\|\bm{v}_D\cdot\bm{n}\|_{0,e}^2,\\
\|q_D\|_{Z}^2&=\|\nabla_hq_D\|_{0,\Omega_D}^2+\sum_{e\in \mathcal{F}_{dl,D}}h_e^{-1}\|\jump{q_D}\|_{0,e}^2,\\
\|\bm{v}_S\|_{1,h}^2&=\|\nabla_h \bm{v}_S\|_{0,\Omega_S}^2+\sum_{e\in \mathcal{F}_{h,S}^0}h_e^{-1}\|\jump{\bm{v}_S}\|_{0,e}^2
\end{align*}
for any $(\bm{v}_D,q_D,\bm{v}_S)\in U_h^D\times P_h^D\times U_h^S$.

Following \cite{Chung09,Chung13}, we have the following inf-sup condition
\begin{align}
\|q_D\|_{Z}\leq C \sup_{\bm{v}_D\in U_h^D}\frac{b_{D}(\bm{v}_D,q_D)}{\|\bm{v}_D\|_{0,\Omega_D}}\quad \forall q_D\in P_h^D.\label{eq:inf-sup}
\end{align}

For later use, we also introduce the following space
\begin{align*}
U_h^{\text{RT}}:&=\{\bm{v}\in H(\text{div};\Omega_S),\bm{v}|_T\in RT_{k-1}, \forall T\in \mathcal{T}_{h,S};\bm{v}\cdot\bm{n}=0\;\mbox{on}\;\Gamma_S\},
\end{align*}
where $H(\text{div};\Omega_S)=\{\bm{v}: \bm{v}\in L^2(\Omega_S)^2, \tdiv \bm{v}\in L^2(\Omega_S)\}$ and $RT_k(T)$ is the Raviart-Thomas element of index $k$ introduced in \cite{RaviartThoms77}.

\begin{remark}
The scheme is locally mass conservative. Indeed, if one choses the test function in \eqref{eq:discrete2} such that $\bm{v}_S=\bm{1}$ on $T$ belonging to $\mathcal{T}_{h,S}$ excluding the elements having the edge lying on $\Gamma$ and zeros otherwise, we have
\begin{align*}
\int_{e}\avg{\underline{\sigma}_h \bm{n}}&=\int_T \bm{f}_S\;dx\quad \forall e\in \mathcal{F}_{h,S}^0 \;\textnormal{and}\;e\subset \partial T.
\end{align*}
In addition, if we take the test function in \eqref{eq:discrete3} such that $q_D=1$ on $D(e)$ for all $e\in \mathcal{F}_{pr,D}^0$ and zeros otherwise, we have
\begin{align*}
\int_{\partial D(e)}\bm{u}_D\cdot\bm{n}&=\int_{D(e)} f_D\;dx\quad \forall e\in \mathcal{F}_{pr,D}^0.
\end{align*}

\end{remark}

\begin{theorem}(unique solvability).
There exists a unique solution to \eqref{eq:discrete1}-\eqref{eq:discrete2}.
\end{theorem}

\begin{proof}

\eqref{eq:discrete1}-\eqref{eq:discrete2} is a square linear system, uniqueness implies existences, thus, it suffices to show the uniqueness.
To prove the uniqueness, we set $\bm{f}_S=\bm{0}$ and $f_D=0$. Then taking $\underline{w}=\underline{\sigma}_h$ and $\bm{v}_S=\bm{u}_{S,h}$ in \eqref{eq:discrete1}-\eqref{eq:discrete2} and summing up the resulting equations yield
\begin{align*}
\|(2\mu)^{-\frac{1}{2}}\mathcal{A}\underline{\sigma}_h\|_{0,\Omega_S}^2+\|K^{-\frac{1}{2}}\bm{u}_{D,h}\|_{0,\Omega_D}^2+\sum_{e\in \mathcal{F}_{h,S}^0}\gamma h_e^{-1}\|\jump{\bm{u}_{S,h}}\|_{0,e}^2
+\frac{1}{G}\sum_{e\in \mathcal{F}_{h,\Gamma}}\|\bm{u}_{S,h}\cdot\bm{t}_S\|_{0,e}^2=0,
\end{align*}
which implies that $\jump{\bm{u}_{S,h}}\mid_e=\bm{0}$ for any $e\in \mathcal{F}_{h,S}^0$, $\mathcal{A}\underline{\sigma}_h=\underline{0}$ and $\bm{u}_{D,h}=\bm{0}$.
Since $\jump{\bm{u}_{S,h}}\mid_e=0$ for any $e\in \mathcal{F}_{h,S}^0$, we have from the definition of $a_S(\cdot,\cdot)$ that
\begin{align*}
a_S(\underline{w},\bm{u}_{S,h})=(\underline{w},\varepsilon_h(\bm{u}_{S,h}))_{\Omega_S}=0.
\end{align*}
Taking $\underline{w}=\varepsilon_h(\bm{u}_{S,h})$ implies that $\varepsilon_h(\bm{u}_{S,h})=\underline{0}$. Then the discrete Korn's inequality (cf. \cite{Brenner03}) gives $\bm{u}_{S,h}=\bm{0}$.

On the other hand we have $\mathcal{A}\underline{\sigma}_h=\underline{0}$, which indicates that we can express $\underline{\sigma}_h$ as
\begin{align*}
\underline{\sigma}_h=\psi I\quad \forall \psi\in P_{k-1}(T).
\end{align*}
Thereby, we can obtain from \eqref{eq:discrete2} that
\begin{align*}
a_S(\underline{\sigma}_h,\bm{v})=\sum_{T\in \mathcal{T}_{h,S}} (\psi,\tdiv\bm{v})_T-\sum_{e\in \mathcal{F}_{h,S}^0}(\jump{\bm{v}\cdot\bm{n}},\avg{\psi})_e=0\quad \forall \bm{v}\in U_h^S.
\end{align*}
Let $\bm{\theta}\in U_h^{\text{RT}}\subset U_h^S$ and set $\bm{v}:=\bm{\theta}$, we have
\begin{align*}
(\psi,\text{div}\,\bm{\theta})_{\Omega_S}=0.
\end{align*}
It follows from the inf-sup condition (cf. \cite{Boffi13}) that
\begin{align*}
\frac{(\psi,\text{div}\,\bm{\theta})_{\Omega_S}}{(\|\bm{\theta}\|_{0,\Omega_S}^2+
\|\text{div}\,\bm{\theta}\|_{0,\Omega_S}^2)^{1/2}}\geq C \|\psi\|_{0,\Omega_S},
\end{align*}
which implies $\psi=0$, thereby $\underline{\sigma}_h=\underline{0}$.

Finally, we have from the inf-sup condition \eqref{eq:inf-sup}, the discrete adjoint property \eqref{eq:adjointD2}, the discrete Poincar\'{e} inequality (cf. \cite{Brenner03}) and \eqref{eq:discrete2} that
\begin{align*}
\|p_{D,h}\|_{0,\Omega_D}\leq C \|p_{D,h}\|_{Z}\leq C\|K^{-1}\bm{u}_{D,h}\|_{0,\Omega_D},
\end{align*}
which implies that $p_{D,h}=0$.
Therefore, the proof is completed.

\end{proof}

\section{A priori error estimates}\label{sec:priori}

In this section we will present the convergence error estimates for all the variables. To begin, we define the following interpolation operators, which play an important role for the subsequent analysis.
%
%We let $\Pi_h^{\text{BDM}}: H(\text{div};\Omega_S)\cap L^p(\Omega_S)\rightarrow U_h^S, p>2$ denote the $BD$ interpolation operator (cf. \cite{Brezzi85}),
%\begin{align*}
%((\bm{v}-\Pi_h^{\text{BDM}}\bm{v})\cdot \bm{n},p_k)_e&=0\quad \forall p_k\in P^k(e) \;\text{and each edge}\; e\subset \partial \tau,\tau\in \mathcal{T}_{h,S},\\
%(\bm{v}-\Pi_h^{\text{BDM}}\bm{v},\nabla p_{k-1})_\tau&=0\quad \forall p_{k-1}\in P^{k-1}(\tau),\tau\in \mathcal{T}_{h,S},\\
%(\bm{v}-\Pi_h^{\text{BDM}}\bm{v},\mbox{curl}\, b)_\tau&=0\quad \forall b\in B^{k+1}(\tau),\tau\in \mathcal{T}_{h,S},
%\end{align*}
%where $B^{k+1}(\tau)=\{p\in P^{k+1}(\tau);\; p\mid_{\partial \tau}=0\}=\lambda_1\lambda_2\lambda_3 P^{k-2}(\tau)$. Here $\lambda_i,\;i=1,2,3$ are the barycentric coordinates of $\tau$.
We let $\Pi_h^{\text{BDM}}$ denote the BDM projection operator (cf. \cite{Brezzi85}),
which satisfies the following error estimates for $0\leq \alpha\leq k+1$ (see, e.g., \cite{Brezzi85,Duran88})
\begin{align*}
\|\bm{v}-\Pi_h^{\text{BDM}}\bm{v}\|_{\alpha,\Omega_S}&\leq C h^{k+1-\alpha}\|\bm{v}\|_{k+1,\Omega_S}\quad \forall \bm{v}\in H^{k+1}(\Omega_S)^2.
%\|\nabla \cdot (\bm{v}-\Pi_h^{\text{BDM}}\bm{v})\|_{0,\Omega_S}&\leq C h^{k}\|\nabla \cdot\bm{v}\|_{k,\Omega_S}\quad \forall \bm{v}\in H^{k}(\text{div};\Omega_S).\label{eq:divBDM}
\end{align*}
Let $\mathcal{I}_h: H^1(\Omega_D)\rightarrow P_h^D$ be defined by
\begin{equation*}
\begin{split}
	(\mathcal{I}_h q-q,\phi)_e
		&=0 \quad \forall \phi\in P^k(e),\forall e\in \mathcal{F}_{pr,D},\\
	(\mathcal{I}_hq-q,\phi)_T
		&=0\quad \forall \phi\in P^{k-1}(T),\forall T\in \mathcal{T}_{h,D}
\end{split}
\end{equation*}
and $\mathcal{J}_h: L^2(\Omega_D)^2\cap H^{1/2+\delta}(\Omega_D)^2\rightarrow U_h^D$, $\delta>0$ be defined by
\begin{equation*}
\begin{split}
	((\mathcal{J}_h\bm{v}-\bm{v})\cdot\bm{n},\varphi)_e
		&=0\quad \forall \varphi\in P^{k}(e),\forall e\in \mathcal{F}_{dl,D},\\
	(\mathcal{J}_h\bm{v}-\bm{v}, \bm{\phi})_T
		&=0\quad \forall \bm{\phi}\in P^{k-1}(T)^2, \forall T\in \mathcal{T}_{h,D}.
\end{split}
\end{equation*}
It is easy to see that $\mathcal{I}_h$ and $\mathcal{J}_h$ are well defined polynomial preserving operators. In addition, the following approximation properties hold for $q\in H^{k+1}(\Omega_D)$ and $\bm{v}\in H^{k+1}(\Omega_D)^2$ (cf. \cite{Ciarlet78,Chung09})
\begin{align}
\|q-\mathcal{I}_hq\|_{\alpha,\Omega_D}&\leq C h^{k+1-\alpha}|q|_{k+1,\Omega_D}\quad 0\leq \alpha\leq k+1,\label{eq:Perror}\\
\|\bm{v}-\mathcal{J}_h\bm{v}\|_{\alpha,\Omega_D}&\leq C h^{k+1-\alpha}|\bm{v}|_{k+1,\Omega_D}\quad 0\leq \alpha\leq k+1\label{eq:uerror}.
\end{align}
Finally, we let $\Pi_h$ denote the $L^2$ orthogonal projection onto $\Sigma_h^S$, then it holds
\begin{align}
\|\underline{w}-\Pi_h\underline{w}\|_{\alpha,\Omega_S}\leq C h^{k-\alpha}\|\underline{w}\|_{k,\Omega_S}\quad \forall \underline{w}\in H^{k}(\Omega_S)^{2\times 2}, 0\leq \alpha\leq k.\label{eq:sigmaerror}
\end{align}

Performing integration by parts on the discrete formulation \eqref{eq:discrete1}-\eqref{eq:discrete3} and using the interface conditions \eqref{eq:interface1}-\eqref{eq:interface32}, we can get the following error equations:
\begin{align}
&((2\mu)^{-1}\mathcal{A}(\underline{\sigma}-\underline{\sigma}_h),\underline{w})_{\Omega_S} = a_S(\underline{w},\bm{u}_S-\bm{u}_{S,h})\quad \underline{w}\in \Sigma_h^S\label{eq:error1},\\
&a_S(\underline{\sigma}-\underline{\sigma}_h,\bm{v}_S)+(K^{-1}(\bm{u}_D-\bm{u}_{D,h}),\bm{v}_D)_{\Omega_D} -b_{D}^*(p_D-p_{D,h},\bm{v}_D)+\sum_{e\in \mathcal{F}_{h,\Gamma}}( \bm{v}_S\cdot\bm{n}_S,p_D-p_{D,h})_e\nonumber\\
&\;+\sum_{e\in \mathcal{F}_{h,S}^0}\gamma h_e^{-1}(\jump{\bm{u}_S-\bm{u}_{S,h}},\jump{\bm{v}_S})_e+\frac{1}{G}\sum_{e\in \mathcal{F}_{h,\Gamma}}((\bm{u}_S-\bm{u}_{S,h})\cdot\bm{t}_S,\bm{v}_S\cdot\bm{t}_S)_e=0\quad \forall \bm{v}_S\in U_h^S\label{eq:error2},\\
&b_{D}(\bm{u}_D-\bm{u}_{D,h},q_D)-\sum_{e\in \mathcal{F}_{h,\Gamma}}((\bm{u}_S-\bm{u}_{S,h})\cdot\bm{n}_S,q_D)_e=0\quad \forall q_D\in P_h^D,\label{eq:error3}
\end{align}
which indicates that our discrete formulation is consistent.

\begin{theorem}\label{thm:sigmaL2}
Let $(\underline{\sigma},\bm{u}_S)\in H^{k}(\Omega_S)^{2\times 2}\times H^{k+1}(\Omega_S)^2$ and $ \bm{u}_D\in H^{k+1}(\Omega_D)^2$. In addition, let $(\underline{\sigma}_h,\bm{u}_{S,h})\in \Sigma_h^S\times U_h^S$ and $\bm{u}_{D,h}\in U_h^D$ be the discrete solution of \eqref{eq:discrete1}-\eqref{eq:discrete3}. Then the following convergence error estimate holds
\begin{align*}
&\|\mu^{-\frac{1}{2}}\mathcal{A}(\underline{\sigma}-\underline{\sigma}_h)\|_{0,\Omega_S}^2
+\|K^{-\frac{1}{2}}(\bm{u}_D-\bm{u}_{D,h})\|_{0,\Omega_D}^2+\sum_{e\in \mathcal{F}_{h,S}^0}\gamma h_e^{-1}\|\jump{\bm{u}_S-\bm{u}_{S,h}}\|_{0,e}^2\\
&\;+\frac{1}{G}\sum_{e\in \mathcal{F}_{h,\Gamma}}\|(\bm{u}_S-\bm{u}_{S,h})\cdot\bm{t}_S\|_{0,e}^2
\leq C \Big(h^{2(k+1)}\|\bm{u}_{D}\|_{k+1,\Omega_D}^2+h^{2k}\|\underline{\sigma}\|_{k,\Omega_S}^2+h^{2k}\|\bm{u}_S\|_{k+1,\Omega_S}^2\Big).
\end{align*}

\end{theorem}

\begin{proof}

Taking $\underline{w}=\Pi_h\underline{\sigma}-\underline{\sigma}_h$, $\bm{v}_S=\Pi_h^{\text{BDM}} \bm{u}_S-\bm{u}_{S,h}$,
$\bm{v}_D=\mathcal{J}_h\bm{u}_D-\bm{u}_{D,h}$ and $q_D=\mathcal{I}_h p_D-p_{D,h}$ in \eqref{eq:error1}-\eqref{eq:error3} and summing up the resulting equations yield
\begin{align*}
&\|(2\mu)^{-\frac{1}{2}}\mathcal{A}(\Pi_h\underline{\sigma}-\underline{\sigma}_h)\|_{0,\Omega_S}^2
+\|K^{-\frac{1}{2}}(\mathcal{J}_h\bm{u}_D-\bm{u}_{D,h})\|_{0,\Omega_D}^2+\sum_{e\in \mathcal{F}_{h,S}^0}\gamma h_e^{-1}\|\jump{\Pi_h^{\text{BDM}}\bm{u}_S-\bm{u}_{S,h}}\|_{0,e}^2\\
&\;+\frac{1}{G}\sum_{e\in \mathcal{F}_{h,\Gamma}}\|(\Pi_h^{\text{BDM}}\bm{u}_S-\bm{u}_{S,h})\cdot\bm{t}_S\|_{0,e}^2=-\sum_{e\in \mathcal{F}_{h,S}^0}\gamma h_e^{-1} (\jump{\bm{u}_S-\Pi_h^{\text{BDM}}\bm{u}_S},\jump{\Pi_h^{\text{BDM}}\bm{u}_S-\bm{u}_{S,h}})_e\\
&\;-\frac{1}{G}\sum_{e\in \mathcal{F}_{h,\Gamma}}((\bm{u}_S-\Pi_h^{\text{BDM}}\bm{u}_S)\cdot\bm{t}_S,(\Pi_h^{\text{BDM}}\bm{u}_S-\bm{u}_{S,h})\cdot\bm{t}_S)_{e}
+a_S(\Pi_h\underline{\sigma}-\underline{\sigma}_h,\bm{u}_S-\Pi_h^{\text{BDM}}\bm{u}_S)\\
&\;-a_S(\underline{\sigma}-\Pi_h\underline{\sigma},\Pi_h^{\text{BDM}}\bm{u}_S-\bm{u}_{S,h})
+(K^{-1}(\mathcal{J}_h\bm{u}_D-\bm{u}_D),\mathcal{J}_h\bm{u}_D-\bm{u}_{D,h})_{\Omega_D}:=\sum_{i=1}^5R_i,
\end{align*}
where we use the definitions of $\mathcal{I}_h$ and $\Pi_h^{\text{BDM}}$, i.e.,
\begin{align*}
\sum_{e\in \mathcal{F}_{h,\Gamma}}((\bm{u}_S-\Pi_h^{\text{BDM}}\bm{u}_{S})\cdot\bm{n}_S,I_hp_D-p_{D,h})_e=0,\\
\sum_{e\in \mathcal{F}_{h,\Gamma}}(\Pi_h^{\text{BDM}}\bm{u}_{S}-\bm{u}_{S,h},p_D-\mathcal{I}_h p_D)_e=0.
\end{align*}
We can bound the first two terms on the right hand side by the Cauchy-Schwarz inequality
\begin{align*}
R_1&\leq C\gamma \Big(\sum_{e\in \mathcal{F}_{h,S}^0}h_e^{-1} \|\jump{\bm{u}_S-\Pi_h^{\text{BDM}}\bm{u}_S}\|_{0,e}^2\Big)^{\frac{1}{2}}
\Big(\sum_{e\in \mathcal{F}_{h,S}^0}h_e^{-1}\|\jump{\Pi_h^{\text{BDM}}\bm{u}_S-\bm{u}_{S,h}}\|_{0,e}^2\Big)^{\frac{1}{2}},\\
R_2&\leq \frac{1}{G}\Big(\sum_{e\in \mathcal{F}_{h,\Gamma}}\|(\bm{u}_S-\Pi_h^{\text{BDM}}\bm{u}_S)\cdot\bm{t}_S\|_{0,e}^2\Big)^{\frac{1}{2}}\Big(\sum_{e\in \mathcal{F}_{h,\Gamma}}\|(\Pi_h^{\text{BDM}}\bm{u}_S-\bm{u}_{S,h})\cdot\bm{t}_S\|_{0,e}^2\Big)^{\frac{1}{2}}.
%R_3&\leq C \|\Pi_h\underline{\sigma}-\underline{\sigma}_h\|_{0,\Omega_S}\|\bm{u}_S-\Pi_h^{\text{BDM}}\bm{u}_S\|_{1,h}.
\end{align*}
Using the definition of $\Pi_h^{\text{BDM}}$, we can infer that
\begin{align*}
a_S(\Pi_h\underline{\sigma}-\underline{\sigma}_h,\bm{u}_S-\Pi_h^{\text{BDM}}\bm{u}_S)&=\sum_{e\in \mathcal{F}_{h,S}^0}(\jump{(\Pi_h\underline{\sigma}-\underline{\sigma}_h)\bm{n}}, \avg{\bm{u}_S-\Pi_h^{\text{BDM}}\bm{u}_S})_e\\
&\;-(\tdiv_h(\Pi_h\underline{\sigma}-\underline{\sigma}_h),\bm{u}_S-\Pi_h^{\text{BDM}}\bm{u}_S)_{\Omega_S}\\
&=\sum_{e\in \mathcal{F}_{h,S}^0}(\jump{(\Pi_h\underline{\sigma}-\underline{\sigma}_h)\bm{n}\cdot\bm{t}}, \avg{(\bm{u}_S-\Pi_h^{\text{BDM}}\bm{u}_S)\cdot\bm{t}})_e\\
&=\sum_{e\in \mathcal{F}_{h,S}^0}(\jump{\mathcal{A}(\Pi_h\underline{\sigma}-\underline{\sigma}_h)\bm{n}\cdot\bm{t}}, \avg{(\bm{u}_S-\Pi_h^{\text{BDM}}\bm{u}_S)\cdot\bm{t}})_e,
\end{align*}
where we use the fact that $ \text{tr}(\Pi_h\underline{\sigma}-\underline{\sigma}_h)\bm{n}\cdot\bm{t}\mid_e=0$ for any $e\in \mathcal{F}_{h,S}^0$.

Thereby, we can estimate $R_3$ by
\begin{equation}
\begin{split}
&a_S(\Pi_h\underline{\sigma}-\underline{\sigma}_h,\bm{u}_S-\Pi_h^{\text{BDM}}\bm{u}_S)\\
&\leq C \|\mathcal{A}(\Pi_h\underline{\sigma}-\underline{\sigma}_h)\|_{0,\Omega_S}\Big(\sum_{e\in \mathcal{F}_{h,S}^0}h_e^{-1}\|\avg{(\bm{u}_S-\Pi_h^{\text{BDM}}\bm{u}_S)\cdot\bm{t}}\|_{0,e}^2\Big)^{\frac{1}{2}}.
\end{split}
\label{eq:R3}
\end{equation}
An application of the Cauchy-Schwarz inequality and the definition of $\Pi_h$ leads to
\begin{align*}
R_4&=-\sum_{e\in \mathcal{F}_{h,S}^0}(\avg{(\underline{\sigma}-\Pi_h\underline{\sigma})\bm{n}},\jump{\Pi_h^{\text{BDM}}\bm{u}_S-\bm{u}_{S,h}})_e
+(\underline{\sigma}-\Pi_h\underline{\sigma}, \nabla_h (\Pi_h^{\text{BDM}}\bm{u}_S-\bm{u}_{S,h}))_{\Omega_S}\\
&=-\sum_{e\in \mathcal{F}_{h,S}^0}(\avg{(\underline{\sigma}-\Pi_h\underline{\sigma})\bm{n}},\jump{\Pi_h^{\text{BDM}}\bm{u}_S-\bm{u}_{S,h}})_e\\
&\leq  \sum_{e\in \mathcal{F}_{h,S}^0}\|\avg{(\underline{\sigma}-\Pi_h\underline{\sigma})\bm{n}}\|_{0,e}\|\jump{\Pi_h^{\text{BDM}}\bm{u}_S-\bm{u}_{S,h}}\|_{0,e}.
%&\leq  \Big(\sum_{e\in \mathcal{F}_{h,s}}h_e\|\avg{(\underline{\sigma}-\Pi_h\underline{\sigma})\bm{n}}\|_{0,e}^2\Big)^{1/2}
%+\epsilon \Big(\sum_{e\in \mathcal{F}_{h,s}}h_e^{-1}\|\jump{\pi_h\bm{u}_S-\bm{u}_{h,S}}\|_{0,e}\Big)^{1/2}
\end{align*}
%
%The inf-sup condition \eqref{eq:inf-sup} and \eqref{eq:error2} imply that
%\begin{align*}
%\|I_hp_D-p_{D,h}\|_Z\leq C \sup_{\bm{v}_D\in U_h^D}\frac{b_{D}^*(I_hp_D-p_{D,h},\bm{v}_D)}{\|\bm{v}_D\|_{0,\Omega_D}}\leq C\|K^{-1}(\bm{u}_D-\bm{u}_{D,h})\|_{0,\Omega_D}.
%\end{align*}
%
%The definition of $\Pi_h^{\text{BDM}}$ implies that
%\begin{align*}
%R_5=0.
%\end{align*}
The Cauchy-Schwarz inequality yields
\begin{align*}
R_5\leq \|K^{-\frac{1}{2}}(\mathcal{J}_h\bm{u}_D-\bm{u}_D)\|_{0,\Omega_D}\|K^{-\frac{1}{2}}(\mathcal{J}_h\bm{u}_D-\bm{u}_{D,h})\|_{0,\Omega_D}.
\end{align*}
%
%Then we can deduce from the trace inequality (cf. \cite{Feng01}) and the discrete Poincar\'{e} inequality (cf. \cite{Brenner03}) that
%\begin{align*}
%\|q\|_{0,\Gamma}\leq C\|q\|_Z.
%\end{align*}
%Consequently, we have
%\begin{align*}
%\sum_{e\in \Gamma_{12}}((\bm{u}_S-\pi_h\bm{u}_{S,h})\cdot\bm{n}_S,I_hp_D-p_{D,h})_e\leq \|\bm{u}_S-\pi_h\bm{u}_{S,h}
%\end{align*}
%
Combining the preceding estimates, Young's inequality,  the interpolation error estimates and the trace inequality implies that
\begin{align*}
&\|\mu^{-\frac{1}{2}}\mathcal{A}(\Pi_h\underline{\sigma}-\underline{\sigma}_h)\|_{0,\Omega_S}^2
+\|K^{-\frac{1}{2}}(\mathcal{J}_h\bm{u}_D-\bm{u}_{D,h})\|_{0,\Omega_D}^2+\sum_{e\in \mathcal{F}_{h,S}^0}\gamma h_e^{-1}\|\jump{\Pi_h^{\text{BDM}}\bm{u}_S-\bm{u}_{S,h}}\|_{0,e}^2\\
&+\frac{1}{G}\sum_{e\in \mathcal{F}_{h,\Gamma}}\|(\Pi_h^{\textnormal{BDM}}\bm{u}_S-\bm{u}_{S,h})\cdot\bm{t}_S\|_{0,e}^2\leq C \Big(h^{2(k+1)}\|\bm{u}_{D}\|_{k+1,\Omega_D}^2+h^{2k}\|\underline{\sigma}\|_{k,\Omega_S}^2+h^{2k}\|\bm{u}_S\|_{k+1,\Omega_S}^2\Big).
\end{align*}
Therefore, the proof is completed by using the interpolation error estimates \eqref{eq:Perror}-\eqref{eq:sigmaerror}.
\end{proof}

We can observe from Theorem~\ref{thm:sigmaL2} that the convergence rate for $L^2$-error of $\bm{u}_D$ is not optimal in terms of the polynomial order. The proof for the optimal convergence rate of $L^2$-error of $\bm{u}_D$ is non-trivial and it relies on a non-standard trace theorem, which will be explained later. In order to achieve the optimal convergence order for $L^2$-error of $\bm{u}_D$, we first need to show the $L^2$ error estimates of $\bm{u}_S$ and $p_D$. To this end, we assume that the following dual problem holds (cf. \cite{Fu18}):
\begin{alignat}{2}
\mbox{div}\,\underline{H}&=\Pi_h^{\text{BDM}}\bm{u}_S-\bm{u}_{S,h}&&\quad \mbox{in}\;\Omega_S,\label{eq:first1-dual}\\
\mathcal{A}\underline{H}&=-2\mu\varepsilon(\bm{\psi}_S)&&\quad \mbox{in}\;\Omega_S\label{eq:first2-dual},\\
\bm{\psi}_S&=\bm{0}&&\quad \mbox{on}\;\Gamma_S
\end{alignat}
and
\begin{alignat}{2}
-\tdiv \bm{\psi}_D&=\mathcal{I}_hp_D-p_{D,h}\quad &&\mbox{in}\; \Omega_D,\label{eq:Darcy1-dual}\\
\bm{\psi}_D-K \nabla \phi&=0&&\mbox{in}\; \Omega_D\label{eq:Darcy2-dual},\\
\phi&=0 &&\mbox{on}\; \Gamma_D.
\end{alignat}
The following interface conditions are prescribed on the interface $\Gamma$
\begin{alignat*}{2}
G\underline{H} \bm{n}_S\cdot\bm{t}_S&=\bm{\psi}_S\cdot\bm{t}_S\quad && \mbox{on}\;\Gamma,\\
\underline{H}\bm{n}_S\cdot\bm{n}_S&=-\phi&& \mbox{on}\;\Gamma,\\
\bm{\psi}_S\cdot\bm{n}_S&=\bm{\psi}_D\cdot\bm{n}_S&& \mbox{on}\;\Gamma.
\end{alignat*}
Assume that the following regularity estimate holds
\begin{align}
\|\bm{\psi}_S\|_{2,\Omega_S}+\|\phi\|_{2,\Omega_D}\leq C \Big(\|\Pi_h^{\text{BDM}}\bm{u}_S-\bm{u}_{S,h}\|_{0,\Omega_S}
+\|\mathcal{I}_hp_D-p_{D,h}\|_{0,\Omega_D}\Big).\label{eq:regularity}
\end{align}

Then we can state the convergence error estimate for $L^2$ errors of $\bm{u}_S$ and $\bm{p}_D$.
\begin{theorem}\label{thm:uL2}
Let $(\underline{\sigma},\bm{u}_S)\in H^{k}(\Omega_S)^{2\times 2}\times H^{k+1}(\Omega_S)^2$ and $ (\bm{u}_D,p_D)\in H^{k+1}(\Omega_D)^2\times H^{k+1}(\Omega_D)$. In addition, let $(\underline{\sigma}_h,\bm{u}_{S,h})\in \Sigma_h^S\times U_h^S$ and $(\bm{u}_{D,h},p_{D,h})\in U_h^D\times p_h^D$ be the discrete solution of \eqref{eq:discrete1}-\eqref{eq:discrete3}. Then the following convergence error estimate holds
\begin{align*}
\|\bm{u}_S-\bm{u}_{S,h}\|_{0,\Omega_S}+\|p_D-p_{D,h}\|_{0,\Omega_D}\leq C h^{k+1} \Big(\|\underline{\sigma}\|_{k,\Omega_S}+\|\bm{u}_S\|_{k+1,\Omega_S}+\|\bm{u}_{D}\|_{k+1,\Omega_D}+\|p_D\|_{k+1,\Omega_D}\Big).
\end{align*}

\end{theorem}

\begin{proof}

Multiplying \eqref{eq:first1-dual} by $\underline{\sigma}-\underline{\sigma}_h$, \eqref{eq:first2-dual}
by $\Pi_h^{\text{BDM}}\bm{u}_S-\bm{u}_{S,h}$, \eqref{eq:Darcy1-dual} by $\mathcal{I}_hp_D-p_{D,h}$ and \eqref{eq:Darcy2-dual} by $\bm{u}_D-\bm{u}_{D,h}$ and performing integration by parts lead to
\begin{align*}
&\|\Pi_h^{\text{BDM}}\bm{u}_S-\bm{u}_{S,h}\|_{0,\Omega_S}^2+\|\mathcal{I}_hp_D-p_{D,h}\|_{0,\Omega_D}^2\\
&=-a_S(\underline{H},\Pi_h^{\text{BDM}}\bm{u}_S-\bm{u}_{S,h})
+((2\mu)^{-1}\mathcal{A}\underline{H},\underline{\sigma}-\underline{\sigma}_h)_{\Omega_S}+a_S(\underline{\sigma}-\underline{\sigma}_h,\bm{\psi}_S)\\
&-b_D(\bm{\psi}_D,\mathcal{I}_hp_D-p_{D,h})+(K^{-1}\bm{\psi}_D,\bm{u}_D-\bm{u}_{D,h})_{\Omega_D}+\sum_{e\in \mathcal{F}_{h,S}^0}h_e^{-1}(\jump{\bm{u}_S-\bm{u}_{S,h}},\jump{\bm{\psi}_S})_e\\
&+b_D^*(\phi, \bm{u}_D-\bm{u}_{D,h})-\sum_{e\in \mathcal{F}_{h,\Gamma}}((\Pi_h^{\text{BDM}}\bm{u}_S-\bm{u}_{S,h})\cdot\bm{n}_S,\phi)_e\\
&\;+\sum_{e\in \mathcal{F}_{h,\Gamma}}(\mathcal{I}_hp_D-p_{D,h},\bm{\psi}_S\cdot\bm{n}_S)_e
+\sum_{e\in \mathcal{F}_{h,\Gamma}}\frac{1}{G}((\Pi_h^{\text{BDM}}\bm{u}_S-\bm{u}_{S,h})\cdot\bm{t}_S,\bm{\psi}_S\cdot\bm{t}_S)_e.
\end{align*}
which, coupled with \eqref{eq:error1}-\eqref{eq:error3} yields
\begin{align*}
&\|\Pi_h^{\text{BDM}}\bm{u}_S-\bm{u}_{S,h}\|_{0,\Omega_S}^2+\|\mathcal{I}_hp_D-p_{D,h}\|_{0,\Omega_D}^2\\
&=a_S(\Pi_h \underline{H}-\underline{H},\Pi_h^{\text{BDM}}\bm{u}_S-\bm{u}_{S,h})
+((2\mu)^{-1}\mathcal{A}(\underline{H}-\Pi_h \underline{H}),\underline{\sigma}-\underline{\sigma}_h)_{\Omega_S}
+a_S(\underline{\sigma}-\underline{\sigma}_h,\bm{\psi}_S-\Pi_h^{\text{BDM}} \bm{\psi}_S)\\
&-b_D(\bm{\psi}_D-\mathcal{J}_h\bm{\psi}_D,\mathcal{I}_hp_D-p_{D,h})+(K^{-1}(\bm{\psi}_D-\mathcal{J}_h\bm{\psi}_D),\bm{u}_D-\bm{u}_{D,h})_{\Omega_D}\\
&+b_D^*(\phi-\mathcal{I}_h\phi, \bm{u}_D-\bm{u}_{D,h})-\sum_{e\in \mathcal{F}_{h,\Gamma}}((\Pi_h^{\text{BDM}}\bm{u}_S-\bm{u}_{S,h})\cdot\bm{n}_S,\phi-\mathcal{I}_h\phi)_e\\
&\;+\sum_{e\in \mathcal{F}_{h,S}^0}h_e^{-1}(\jump{\bm{u}_S-\bm{u}_{S,h}},\jump{\bm{\psi}_S-\Pi_h^{\text{BDM}}\bm{\psi}_S})_e
+\sum_{e\in \mathcal{F}_{h,\Gamma}}(\mathcal{I}_hp_D-p_{D,h},(\bm{\psi}_S-\Pi_h^{\text{BDM}}\bm{\psi}_S)\cdot\bm{n}_S)_e\\
&\;+\frac{1}{G}\sum_{e\in \mathcal{F}_{h,\Gamma}}((\Pi_h^{\text{BDM}}\bm{u}_S-\bm{u}_{S,h})\cdot\bm{t}_S,(\bm{\psi}_S-\Pi_h^{\text{BDM}}\bm{\psi}_S)\cdot\bm{t}_S)_e.
\end{align*}
The first two terms on the right hand side can be bounded by the Cauchy-Schwarz inequality
\begin{align*}
a_S(\Pi_h \underline{H}-\underline{H},\Pi_h^{\text{BDM}}\bm{u}_S-\bm{u}_{S,h})
%&\leq \Big(\sum_{e\in \mathcal{F}_{h,S}}h_e\|\avg{\Pi_h H-H}\|_{0,e}^2\Big)^{1/2}
%\Big(\sum_{e\in \mathcal{F}_{h,S}}h_e^{-1}\|\jump{\Pi_h^{\text{BDM}}\bm{u}_S-\bm{u}_{S,h}}\|_{0,e}^2\Big)^{1/2}\\
&\leq C h\|\underline{H}\|_{1,\Omega_S}\Big(\sum_{e\in \mathcal{F}_{h,S}^0}h_e^{-1}\|\jump{\Pi_h^{\text{BDM}}\bm{u}_S-\bm{u}_{S,h}}\|_{0,e}^2\Big)^{1/2},\\
((2\mu)^{-1}\mathcal{A}(\underline{H}-\Pi_h \underline{H}),\underline{\sigma}-\underline{\sigma}_h)_{\Omega_S}&\leq \|\mu^{-\frac{1}{2}}\mathcal{A}(\underline{\sigma}-\underline{\sigma}_h)\|_{0,\Omega_S}\|\mu^{-\frac{1}{2}}\mathcal{A}(\underline{H}-\Pi_h \underline{H})\|_{0,\Omega_S}\\
&\leq C h\|\mu^{-\frac{1}{2}}\mathcal{A}(\underline{\sigma}-\underline{\sigma}_h)\|_{0,\Omega_S}\|\underline{H}\|_{1,\Omega_S}.
\end{align*}
%\begin{align*}
%(2\mu^{-1}A(H-\Pi_h H),\underline{\sigma}-\underline{\sigma}_h)\leq \|\mathcal{A}(\underline{\sigma}-\underline{\sigma}_h)\|_{0,\Omega_S}\|2\mu^{-1}A(H-\Pi_h H)\|_{0,\Omega_S}
%\end{align*}
We can rewrite the third term as follows
\begin{equation}
\begin{split}
a_S(\underline{\sigma}-\underline{\sigma}_h,\bm{\psi}_S-\Pi_h^{\text{BDM}} \bm{\psi}_S)
&=a_S(\underline{\sigma}-\Pi_h\underline{\sigma},\bm{\psi}_S-\Pi_h^{\text{BDM}} \bm{\psi}_S)\\
&\;+
a_S(\Pi_h\underline{\sigma}-\underline{\sigma}_h,\bm{\psi}_S-\Pi_h^{\text{BDM}} \bm{\psi}_S),
\end{split}
\label{eq:as}
\end{equation}
where the first term on the right hand side can be bounded by
\begin{align*}
a_S(\underline{\sigma}-\Pi_h\underline{\sigma},\bm{\psi}_S-\Pi_h^{\text{BDM}} \bm{\psi}_S)\leq C h^{k+1}\|\underline{\sigma}\|_{k,\Omega_S}\|\bm{\psi}_S\|_{2,\Omega_S}.
\end{align*}
The second term on the right hand side of \eqref{eq:as} can be estimated similarly to \eqref{eq:R3}. Indeed, we have
\begin{align*}
&a_S(\Pi_h\underline{\sigma}-\underline{\sigma}_h,\bm{\psi}_S-\Pi_h^{\text{BDM}} \bm{\psi}_S)\\
&\leq \|\mathcal{A}(\Pi_h\underline{\sigma}-\underline{\sigma}_h)\|_{0,\Omega_S}\Big(\sum_{e\in \mathcal{F}_{h,S}^0}h_e^{-1}\|\avg{(\bm{\psi}_S-\Pi_h^{\text{BDM}}\bm{\psi}_S)\cdot\bm{t}}\|_{0,e}^2\Big)^{1/2}\\
&\leq C h\|\mathcal{A}(\Pi_h\underline{\sigma}-\underline{\sigma}_h)\|_{0,\Omega_S}\|\bm{\psi}_S\|_{2,\Omega_S}.
\end{align*}
On the other hand, we have from the definitions of $\mathcal{I}_h$, $\mathcal{J}_h$ and $\Pi_h^{\text{BDM}}$
\begin{align*}
\sum_{e\in \mathcal{F}_{h,\Gamma}}((\Pi_h^{\text{BDM}}\bm{u}_S-\bm{u}_{S,h})\cdot\bm{n}_S,\phi-\mathcal{I}_h\phi)_e&=0,\\
\sum_{e\in \mathcal{F}_{h,\Gamma}}(I_hp_D-p_{D,h},(\bm{\psi}_S-\Pi_h^{\text{BDM}}\bm{\psi}_S)\cdot\bm{n}_S)_e&=0,\\
b_D(\bm{\psi}_D-\mathcal{J}_h\bm{\psi}_D,\mathcal{I}_hp_D-p_{D,h})&=0,\\
b_D^*(\phi-\mathcal{I}_h\phi, \bm{u}_D-\bm{u}_{D,h})&=b_D^*(\phi-\mathcal{I}_h\phi, \bm{u}_D-\mathcal{J}_h\bm{u}_{D})\leq \|\phi-\mathcal{I}_h\phi\|_Z \|\bm{u}_D-\mathcal{J}_h\bm{u}_{D}\|_{0,h}\\
&\leq C h^{k+1}\|\phi\|_{2,\Omega_D}\|\bm{u}_D\|_{k+1,\Omega_D}.
\end{align*}
%\begin{align*}
%b_D^*(\phi-I_h\phi, \bm{u}_D-\bm{u}_{D,h})&=b_D^*(\phi-I_h\phi, \bm{u}_D-J_h\bm{u}_{D})\leq \|\phi-I_h\phi\|_Z \|\bm{u}_D-J_h\bm{u}_{D}\|_{0,h}\\
%&\leq C h^{k+1}\|\phi\|_{2,\Omega_D}\|\bm{u}_D\|_{k+1,\Omega_D}.
%\end{align*}
The Cauchy-Schwarz inequality implies
\begin{align*}
&\sum_{e\in \mathcal{F}_{h,S}^0}h_e^{-1}(\jump{\bm{u}_S-\bm{u}_{S,h}},\jump{\bm{\psi}_S-\Pi_h^{\text{BDM}}\bm{\psi}_S})_e\\
&\leq \Big(\sum_{e\in \mathcal{F}_{h,S}^0}h_e^{-1}\|\jump{\bm{u}_S-\bm{u}_{S,h}}\|_{0,e}^2\Big)^{\frac{1}{2}}\Big(\sum_{e\in \mathcal{F}_{h,S}^0}h_e^{-1}\|\jump{\bm{\psi}_S-\Pi_h^{\text{BDM}}\bm{\psi}_S}\|_{0,e}^2\Big)^{\frac{1}{2}}\\
&\leq C h \Big(\sum_{e\in \mathcal{F}_{h,S}^0}h_e^{-1}\|\jump{\bm{u}_S-\bm{u}_{S,h}}\|_{0,e}^2\Big)^{\frac{1}{2}}\|\bm{\psi}_S\|_{2,\Omega_S}.
\end{align*}
The last term can be bounded by the trace inequality and \eqref{eq:regularity}
\begin{align*}
&\frac{1}{G}\sum_{e\in \mathcal{F}_{h,\Gamma}}((\Pi_h^{\text{BDM}}\bm{u}_S-\bm{u}_{S,h})\cdot\bm{t}_S,(\bm{\psi}_S-\Pi_h^{\text{BDM}}\bm{\psi}_S)\cdot\bm{t}_S)_e\\
&\leq \frac{1}{G}\Big(\sum_{e\in \mathcal{F}_{h,\Gamma}}\|(\Pi_h^{\text{BDM}}\bm{u}_S-\bm{u}_{S,h})\cdot\bm{t}_S\|_{0,e}^2\Big)^{1/2}\Big(\sum_{e\in \mathcal{F}_{h,\Gamma}}\|(\bm{\psi}_S-\Pi_h^{\text{BDM}}\bm{\psi}_S)\cdot\bm{t}_S\|_{0,e}^2\Big)^{1/2}\\
&\leq C h\|\Pi_h^{\text{BDM}}\bm{u}_S-\bm{u}_{S,h}\|_{0,\Omega_S}\|(\Pi_h^{\text{BDM}}\bm{u}_S-\bm{u}_{S,h})\cdot\bm{t}_S\|_{0,\Gamma}.
\end{align*}
Then combining the preceding estimates with Theorem~\ref{thm:sigmaL2} and the interpolation error estimates \eqref{eq:Perror}-\eqref{eq:sigmaerror} implies the desired estimate.

\end{proof}

Proceeding similarly to Lemma~2.3 of \cite{Cai17}, we can get the following lemma.
\begin{lemma}\label{lemma:trace}
Let $e$ be an edge of $T\in \mathcal{T}_h$ and let $\delta>0$. Let $
V^{1+\alpha}(T):=\{\bm{\psi}\in H^{1+\alpha}(T)^2; \tdiv \bm{\psi}\in L^2(T)\}$ for $\alpha>0$ and $T\in \mathcal{T}_h$.
Assume that $\bm{\psi}$ is a given function in $V^{1+s}(T)$. Then, there exists a small $0<s<\min\{\delta,1/2\}$ and a positive constant $C$ independent of $s$ such that
\begin{align*}
\|\bm{\psi}\cdot\bm{n}\|_{s-1/2,e}\leq C \Big(\|\bm{\psi}\|_{0,T}+h_T\|\tdiv \bm{\psi}\|_{0,T}\Big).
\end{align*}

\end{lemma}

The trace inequality introduced in Lemma~\ref{lemma:trace} is crucial to obtain the optimal convergence error estimate for $L^2$-error of $\bm{u}_D$. Now, we can state the following theorem.

\begin{theorem}
Let $(\underline{\sigma},\bm{u}_S)\in H^{k}(\Omega_S)^{2\times 2}\times H^{k+1}(\Omega_S)^2$ and $ \bm{u}_D\in H^{k+1}(\Omega_D)^2$. In addition, let $(\underline{\sigma}_h,\bm{u}_{S,h})\in \Sigma_h^S\times U_h^S$ and $\bm{u}_{D,h}\in U_h^D$ be the discrete solution of \eqref{eq:discrete1}-\eqref{eq:discrete3}. Then, we have
\begin{align*}
\|K^{-\frac{1}{2}}(\bm{u}_D-\bm{u}_{D,h})\|_{0,\Omega_D}\leq C h^{k+1}\Big(\|\bm{u}_{D}\|_{k+1,\Omega_D}+\|\underline{\sigma}\|_{k,\Omega_S}+\|\bm{u}_S\|_{k+1,\Omega_S}+\|p_D\|_{k+1,\Omega_D}\Big).
\end{align*}

\end{theorem}

\begin{proof}

Taking $\bm{v}_S=\bm{0}$ in \eqref{eq:error2}, we have
\begin{align}
(K^{-1}(\bm{u}_D-\bm{u}_{D,h}),\bm{v}_D)_{\Omega_D}-b_{D}^*(p_D-p_{D,h},\bm{v}_D)&=0,\\
b_{D}(\mathcal{J}_h\bm{u}_D-\bm{u}_{D,h},q_D)-\sum_{e\in \mathcal{F}_{h,\Gamma}}((\Pi_h^{\text{BDM}}\bm{u}_S-\bm{u}_{S,h})\cdot\bm{n}_S,q_D)_e&=0\label{eq:b2}.
\end{align}
Setting $\bm{v}_D=\mathcal{J}_h\bm{u}_D-\bm{u}_{D,h}$ and $q_D=\mathcal{I}_hp_D-p_{D,h}$ in the above equations yields
\begin{align*}
(K^{-1}(\bm{u}_D-\bm{u}_{D,h}),\mathcal{J}_h\bm{u}_D-\bm{u}_{D,h})_{\Omega_D}
-\sum_{e\in \mathcal{F}_{h,\Gamma}}((\Pi_h^{\text{BDM}}\bm{u}_S-\bm{u}_{S,h})\cdot\bm{n}_S,\mathcal{I}_hp_D-p_{D,h})_e&=0,
\end{align*}
where we make use of the definitions of the interpolation operators $\mathcal{I}_h$ and $\mathcal{J}_h$.

For any $e\in \mathcal{F}_{h,\Gamma}$ and $0<s <\frac{1}{2}$, we have from the inverse inequality and Lemma~\ref{lemma:trace} that
\begin{align*}
((\Pi_h^{\text{BDM}}\bm{u}_S-\bm{u}_{S,h})\cdot\bm{n}_S,I_hp_D-p_{D,h})_{e}&\leq C \|(\Pi_h^{\text{BDM}}\bm{u}_S-\bm{u}_{S,h})\cdot\bm{n}_S\|_{s-\frac{1}{2},e}\|I_hp_D-p_{D,h}\|_{\frac{1}{2}-s,e}\\
&\leq C\Big(\|\Pi_h^{\text{BDM}}\bm{u}_S-\bm{u}_{S,h}\|_{0,T_1}\\
&\;+h_{T_1}\|\nabla \cdot (\Pi_h^{\text{BDM}}\bm{u}_S-\bm{u}_{S,h})\|_{0,T_1}\Big)h_{T_2}^{s-\frac{1}{2}}\|I_hp_D-p_{D,h}\|_{0,e},
\end{align*}
where $T_1\in \mathcal{T}_{h,S}$ and $T_2\in \mathcal{T}_{h,D}$.

For $s$ sufficiently close to $\frac{1}{2}$, it holds $h_{T_2}^{s-\frac{1}{2}}\leq 2$, thereby it follows that
\begin{equation}
\begin{split}
&((\Pi_h^{\text{BDM}}\bm{u}_S-\bm{u}_{S,h})\cdot\bm{n}_S,I_hp_D-p_{D,h})_{0,e}\leq \Big(\|\Pi_h^{\text{BDM}}\bm{u}_S-\bm{u}_{S,h}\|_{0,T_1}\\
&\;+h_{T_1}\|\nabla \cdot (\Pi_h^{\text{BDM}}\bm{u}_S-\bm{u}_{S,h})\|_{0,T_1}\Big)\|I_hp_D-p_{D,h}\|_{0,e}.
\end{split}
\label{eq:uSh1}
\end{equation}
On the other hand, we have from the trace inequality (cf. \cite{Feng01}) and the discrete Poincar\'{e} inequality (cf. \cite{Brenner03}) that
\begin{align*}
\|I_hp_D-p_{D,h}\|_{0,\Gamma}\leq C \|I_hp_D-p_{D,h}\|_Z.
\end{align*}
In addition, it follows from the inf-sup condition \eqref{eq:inf-sup} and \eqref{eq:error2} that
\begin{align*}
\|I_hp_D-p_{D,h}\|_{Z}\leq C\|K^{-1}(\bm{u}_D-\bm{u}_{D,h})\|_{0,\Omega_D}.
\end{align*}
Thereby, summing \eqref{eq:uSh1} over all the edges in $\mathcal{F}_{h,\Gamma}$ and using the above arguments lead to
\begin{align*}
\|K^{-\frac{1}{2}}(\mathcal{J}_h\bm{u}_D-\bm{u}_{D,h})\|_{0,\Omega_D}\leq \|K^{-\frac{1}{2}}(\mathcal{J}_h\bm{u}_D-\bm{u}_D)\|_{0,\Omega_D}+\|\Pi_h^{\text{BDM}}\bm{u}_S-\bm{u}_{S,h}\|_{0,\Omega_S},
\end{align*}
which coupled with the interpolation error estimates \eqref{eq:Perror}-\eqref{eq:sigmaerror} and Theorem~\ref{thm:uL2} completes the proof.

\end{proof}

\begin{theorem}

Let $(\bm{u}_{S,h},\bm{u}_{D,h})\in U_h^S\times U_h^D$ be the numerical solution of \eqref{eq:discrete1}-\eqref{eq:discrete3}, then the interface condition \eqref{eq:interface1} is satisfied exactly at the discrete level, i.e.,
\begin{align*}
\bm{u}_{S,h}\cdot\bm{n}_S=\bm{u}_{D,h}\cdot\bm{n}_S\quad \textnormal{on}\;\Gamma.
\end{align*}

\end{theorem}

\begin{proof}

We can infer from \eqref{eq:b2} and the discrete adjoint property that
\begin{align*}
&b_{D}^*(q_D,\mathcal{J}_h\bm{u}_D-\bm{u}_{D,h})-\sum_{e\in \mathcal{F}_{h,\Gamma}}((\Pi_h^{\text{BDM}}\bm{u}_S-\bm{u}_{S,h})\cdot\bm{n}_S,q_D)_e=0\quad \forall q_D\in P_h^D,
\end{align*}
thereby, we have
\begin{align*}
&-\sum_{e\in \mathcal{F}_{pr,D}^0}(q_D,\jump{(\mathcal{J}_h\bm{u}_D-\bm{u}_{D,h})\cdot\bm{n}})_e
+(q_D,\tdiv_h (\mathcal{J}_h\bm{u}_D-\bm{u}_{D,h}))_{\Omega_D}-\sum_{e\in \mathcal{F}_{h,\Gamma}}((\mathcal{J}_h\bm{u}_D-\bm{u}_{D,h})\cdot\bm{n}_D,q_{D})_e\\
&\;-\sum_{e\in \mathcal{F}_{h,\Gamma}}((\Pi_h^{\text{BDM}}\bm{u}_S-\bm{u}_{S,h})\cdot\bm{n}_S,q_D)_e=0.
\end{align*}
Then we can define $q_D$ using the degrees of freedom defined in \eqref{eq:dof1}-\eqref{eq:dof2} such that it satisfies
\begin{align*}
&\sum_{e\in \mathcal{F}_{pr,D}^0}\|\jump{(\mathcal{J}_h\bm{u}_D-\bm{u}_{D,h})\cdot\bm{n}}\|_{0,e}^2+\sum_{e\in \mathcal{F}_{h,\Gamma}}\|(\Pi_h^{\text{BDM}}\bm{u}_S-\bm{u}_{S,h}-(\mathcal{J}_h\bm{u}_D-\bm{u}_{D,h}))\cdot\bm{n}_D\|_{0,e}^2\\
&\;+\|\tdiv_h (\mathcal{J}_h\bm{u}_D-\bm{u}_{D,h})\|_{0,\Omega_D}^2=0.
\end{align*}
Therefore, we have
\begin{align*}
\jump{(\mathcal{J}_h\bm{u}_D-\bm{u}_{D,h})\cdot\bm{n}}\mid_e&=0\quad \forall e\in \mathcal{F}_{pr,D}^0,\\
\tdiv (\mathcal{J}_h\bm{u}_D-\bm{u}_{D,h})|_T&=0\quad \forall T\in \mathcal{T}_{h,D},\\
(\Pi_h^{\text{BDM}}\bm{u}_S-\bm{u}_{S,h})\cdot \bm{n}_D&=(\mathcal{J}_h\bm{u}_D-\bm{u}_{D,h})\cdot\bm{n}_D\quad \mbox{on}\;\Gamma.
\end{align*}
We can infer that
\begin{align*}
\bm{u}_{S,h}\cdot\bm{n}_S=\bm{u}_{D,h}\cdot\bm{n}_S\quad \mbox{on}\;\Gamma.
\end{align*}
Therefore, the proof is completed.
%Using the above property, we can rewrite \eqref{eq:b2} as
%\begin{align*}
%-\sum_{e\in \Gamma_{12}}( (\mathcal{J}_h\bm{u}_D-\bm{u}_{h,D})\cdot\bm{n}_D,p_{D})_e-((\Pi_h^{\text{BDM}}\bm{u}_S-\bm{u}_{h,S})\cdot\bm{n}_S,q_D)_\Gamma=0
%\end{align*}
%
%
%\begin{align*}
%b_{D}^*(p_D,\bm{v}_D)&=-\sum_{e\in \mathcal{F}_{pr,D}}(p_D,\jump{\bm{v}_D\cdot\bm{n}})_e+(p_D,\tdiv \bm{v}_D)-\sum_{e\in \Gamma_{12}}( \bm{v}_D\cdot\bm{n}_D,p_{D})_e,\\
%b_{D}(\bm{u}_D,q_D)&=\sum_{e\in \mathcal{F}_{dl,D}}(\bm{u}_D\cdot\bm{n},\jump{q_D})_e-(\bm{u}_D,\nabla_h q_D)
%\end{align*}

\end{proof}

\section{Robin type domain decomposition method}\label{sec:robin}

In this section, we introduce a novel domain decomposition method based on the Robin-type interface conditions, which can decouple the global system \eqref{eq:discrete1}-\eqref{eq:discrete3} into the Stokes subproblem and the Darcy subproblem. To begin with, we define two functions $g_D$ and $g_S$ on the interface $\Gamma$ such that
\begin{align}
-\underline{\sigma}\bm{n}_{S}\cdot \bm{n}_{S}-\delta_f \bm{u}_S\cdot\bm{n}_S&=g_S,\label{eq:robin1}\\
p_D-\delta_p \bm{u}_D\cdot\bm{n}_D&=g_D,\label{eq:robin2}
\end{align}
where $\delta_f$ and $\delta_p$ are two positive constants which will be specified later.

In addition, the following compatibility conditions should be satisfied in order to get the equivalence between the interface conditions \eqref{eq:interface1}-\eqref{eq:interface22} and \eqref{eq:robin1}-\eqref{eq:robin2}
\begin{align}
g_S&=p_D(1+\frac{\delta_f}{\delta_p})-g_{D} \frac{\delta_f}{\delta_p}\quad\; \mbox{on}\;\Gamma,\label{eq:compatibility1}\\
g_D &= g_S+(\delta_f+\delta_p)\bm{u}_S\cdot\bm{n}_S\quad \mbox{on}\;\Gamma.\label{eq:compatibility2}
\end{align}
Indeed, we have from \eqref{eq:compatibility1}-\eqref{eq:compatibility2}
\begin{align*}
g_D& =p_D(1+\frac{\delta_f}{\delta_p})-g_{D} \frac{\delta_f}{\delta_p}+(\delta_f+\delta_p)\bm{u}_S\cdot\bm{n}_S\\
&=p_D+\delta_f \bm{u}_D\cdot\bm{n}_D+(\delta_f+\delta_p)\bm{u}_S\cdot\bm{n}_S.
\end{align*}
Then it follows from \eqref{eq:robin2} that
\begin{align*}
(\delta_f+\delta_p)\bm{u}_D\cdot\bm{n}_D+(\delta_f+\delta_p)\bm{u}_S\cdot\bm{n}_S=0,
\end{align*}
which implies
\begin{align*}
\bm{u}_D\cdot\bm{n}_D=-\bm{u}_S\cdot\bm{n}_S.
\end{align*}
On the other hand, we can deduce from \eqref{eq:robin1}, \eqref{eq:robin2} and \eqref{eq:compatibility2} that
\begin{align*}
-\underline{\sigma}\bm{n}_{S}\cdot \bm{n}_{S}=\delta_f \bm{u}_S\cdot\bm{n}_S+g_S=g_D-\delta_p\bm{u}_S\cdot\bm{n}_S=p_D.
\end{align*}
Then we can propose the coupled discrete formulation with the modified interface conditions: For given $g_{D,h},g_{S,h}$, find $(\widetilde{\underline{\sigma}}_h,\widetilde{\bm{u}}_{S,h})\in \Sigma_h^S\times U_h^S$ and $(\widetilde{\bm{u}}_{D,h},\widetilde{p}_{D,h})\in U_h^D\times P_h^D$ such that
\begin{align}
&((2\mu)^{-1}\mathcal{A}\widetilde{\underline{\sigma}}_h,\underline{w})_{\Omega_S} = a_S(\underline{w},\widetilde{\bm{u}}_{S,h})\quad \underline{w}\in \Sigma_h^S,\label{eq:discrete12}\\
&a_S(\widetilde{\underline{\sigma}}_h,\bm{v}_S)+\sum_{e\in \mathcal{F}_{h,S}^0}\gamma h_e^{-1}(\jump{\widetilde{\bm{u}}_{S,h}},\jump{\bm{v}_S})_e+\delta_f\sum_{e\in \mathcal{F}_{h,\Gamma}}(\widetilde{\bm{u}}_{S,h}\cdot\bm{n}_S,\bm{v}_S\cdot\bm{n}_S)_{e}
+\frac{1}{G}\sum_{e\in \mathcal{F}_{h,\Gamma}}(\widetilde{\bm{u}}_{S,h}\cdot\bm{t}_S,\bm{v}_S\cdot\bm{t}_S)_e\nonumber\\
&+(K^{-1}\widetilde{\bm{u}}_{D,h},\bm{v}_D)_{\Omega_D}-b_{D}^*(\widetilde{p}_{D,h},\bm{v}_D)
=(\bm{f}_S,\bm{v}_S)_{\Omega_S}-\sum_{e\in \mathcal{F}_{h,\Gamma}}(g_{S,h},\bm{v}_S\cdot\bm{n}_S)_{e}\quad \forall \bm{v}_S\in U_h^S,\label{eq:discrete22}\\
&b_{D}(\widetilde{\bm{u}}_{D,h},q_D)+\sum_{e\in \mathcal{F}_{h,\Gamma}}\frac{1}{\delta_p}(\widetilde{p}_{D,h},q_D)_e=(f_D,q_D)_{\Omega_D}+\sum_{e\in \mathcal{F}_{h,\Gamma}}\frac{1}{\delta_p}(g_{D,h},q_D)_e\quad \forall q_D\in P_h^D\label{eq:discrete32}.
\end{align}
%which are supplemented with the following compatibility conditions
%\begin{alignat*}{2}
%g_{S,h}&=\widetilde{p}_{D,h}(1+\frac{\delta_f}{\delta_p})-g_{D,h}\frac{\delta_f}{\delta_p}&&\quad \mbox{on}\;\Gamma,\\
%g_{D,h} &= g_{S,h}+(\delta_f+\delta_p)\bm{u}_{S,h}\cdot\bm{n}_S&&\quad \mbox{on}\;\Gamma.
%\end{alignat*}
%Proceeding similarly to Lemma~5.1 of \cite{Chendd11}, we can obtain the following lemma.
%\begin{lemma}
%Let $(\underline{\sigma}_h,\bm{u}_{S,h},\bm{u}_{D,h},p_{D,h})\in \Sigma_h^S\times U_h^S\times U_h^D\times P_h^D$ be the solution of \eqref{eq:discrete1}-\eqref{eq:discrete3}, and $(\widetilde{\underline{\sigma}}_h,\widetilde{\bm{u}}_{S,h},\widetilde{\bm{u}}_{D,h},\widetilde{p}_{D,h})\in \Sigma_h^S\times U_h^S\times U_h^D\times P_h^D$ be the solution of \eqref{eq:discrete12}-\eqref{eq:discrete32}. Then
%\begin{align*}
%(\underline{\sigma}_h,\bm{u}_{S,h},\bm{u}_{D,h},p_{D,h})&=
%(\widetilde{\underline{\sigma}}_h,\widetilde{\bm{u}}_{S,h},\widetilde{\bm{u}}_{D,h},\widetilde{p}_{D,h})
%%(\bm{u}_{D,h},p_{D,h})&=(\widetilde{\bm{u}}_{D,h},\widetilde{p}_{D,h})
%\end{align*}
%if an only if the following compatibility conditions hold
%\begin{alignat*}{2}
%g_{S,h}&=\widetilde{p}_{D,h}(1+\frac{\delta_f}{\delta_p})-g_{D,h}\frac{\delta_f}{\delta_p}&&\quad \textnormal{on}\;\Gamma,\\
%g_{D,h} &= g_{S,h}+(\delta_f+\delta_p)\widetilde{\bm{u}}_{S,h}\cdot\bm{n}_S&&\quad \textnormal{on}\;\Gamma.
%\end{alignat*}
%
%\end{lemma}

\begin{lemma}

The discrete formulations \eqref{eq:discrete1}-\eqref{eq:discrete3} are equivalent to \eqref{eq:discrete12}-\eqref{eq:discrete32} if and only if
the following conditions hold on the interface $\Gamma$, i.e.,
\begin{align}
\widetilde{p}_{D,h}-\delta_p \widetilde{\bm{u}}_{S,h}\cdot\bm{n}_D&=g_{D,h},\label{eq:compa1}\\
\widetilde{p}_{D,h}-\delta_f \widetilde{\bm{u}}_{S,h}\cdot\bm{n}_S&=g_{S,h}.\label{eq:compa2}
\end{align}

\end{lemma}

\begin{proof}

If the discrete formulations \eqref{eq:discrete1}-\eqref{eq:discrete3} are equivalent to \eqref{eq:discrete12}-\eqref{eq:discrete32}, we can infer from \eqref{eq:discrete3} and \eqref{eq:discrete32} that
\begin{align*}
\frac{1}{\delta_p}(\widetilde{p}_{D,h}-g_{D,h},q_D)_\Gamma=-(\widetilde{\bm{u}}_{S,h}\cdot\bm{n}_S,q_D)_\Gamma\quad \forall q_D\in P_h^D.
\end{align*}
which implies \eqref{eq:compa1}. Similarly, \eqref{eq:discrete2} and \eqref{eq:discrete22} imply \eqref{eq:compa2}. The opposite direction can be proved in a similar manner.

\end{proof}

The corresponding decoupled formulations with Robin-type interface conditions read as follows: For given $g_{D,h},g_{S,h}$, find $(\widetilde{\underline{\sigma}}_h,\widetilde{\bm{u}}_{S,h})\in \Sigma_h^S\times U_h^S$ and $(\widetilde{\bm{u}}_{D,h},\widetilde{p}_{D,h})\in U_h^D\times P_h^D$ such that
\begin{align}
&(K^{-1}\widetilde{\bm{u}}_{D,h},\bm{v}_D)_{\Omega_D}-b_{D}^*(\widetilde{p}_{D,h},\bm{v}_D)=0\quad \forall \bm{v}_D\in U_h^D,\label{eq:Darcy-robin1}\\
&b_{D}(\widetilde{\bm{u}}_{D,h},q_D)+\sum_{e\in \mathcal{F}_{h,\Gamma}}\frac{1}{\delta_p}(\widetilde{p}_{D,h},q_D)_e=(f_D,q_D)_{\Omega_D}+\sum_{e\in \mathcal{F}_{h,\Gamma}}\frac{1}{\delta_p}(g_{D,h},q_D)_e\quad \forall q_D\in P_h^D\label{eq:Darcy-robin2},\\
&((2\mu)^{-1}\mathcal{A}\widetilde{\underline{\sigma}}_h,\underline{w})_{\Omega_S} = a_S(\underline{w},\widetilde{\bm{u}}_{S,h})\quad \forall \underline{w}\in \Sigma_h^S,\label{eq:Stokes-robin1}\\
&a_S(\widetilde{\underline{\sigma}}_h,\bm{v}_S)+\sum_{e\in \mathcal{F}_{h,S}^0}\gamma h_e^{-1}(\jump{\widetilde{\bm{u}}_{S,h}},\jump{\bm{v}_S})_e+\sum_{e\in \mathcal{F}_{h,\Gamma}}\delta_f(\widetilde{\bm{u}}_{S,h}\cdot\bm{n}_S,\bm{v}_S\cdot\bm{n}_S)_{e}\nonumber\\
&+\frac{1}{G}\sum_{e\in \mathcal{F}_{h,\Gamma}}(\widetilde{\bm{u}}_{S,h}\cdot\bm{t}_S,\bm{v}_S\cdot\bm{t}_S)_e=(\bm{f}_S,\bm{v}_S)_{\Omega_S}-\sum_{e\in \mathcal{F}_{h,\Gamma}}(g_{S,h},\bm{v}_S\cdot\bm{n}_S)_{e}\quad \forall \bm{v}_S\in U_h^S,\label{eq:Stokes-robin2}
\end{align}
which are supplemented with the following compatibility conditions
\begin{alignat*}{2}
g_{S,h}&=\widetilde{p}_{D,h}(1+\frac{\delta_f}{\delta_p})-g_{D,h}\frac{\delta_f}{\delta_p}&&\quad \mbox{on}\;\Gamma,\\
g_{D,h} &= g_{S,h}+(\delta_f+\delta_p)\widetilde{\bm{u}}_{S,h}\cdot\bm{n}_S&&\quad \mbox{on}\;\Gamma.
\end{alignat*}

We can now present the domain decomposition method based on the modified decoupled formulations.

\noindent\textbf{Algorithm DDM}

 \begin{itemize}
        \item[Step 1.] Initial values of $g_{S,h}^0$ and $g_{D,h}^0$ are defined.
        \item[Step 2.] For $m=0,1,2,\ldots,$ solve the following Stokes and Darcy systems independently. Specifically, find $(\bm{u}_{D,h}^m,p_{D,h}^m)\in U_h^D\times P_h^D$ such that
\begin{align}
(K^{-1}\bm{u}_{D,h}^m,\bm{v}_D)_{\Omega_D}-b_{D}^*(p_{D,h}^m,\bm{v}_D)&=0\quad \forall \bm{v}_D\in U_h^D,\label{eq:Darcy-robink1}\\
b_{D}(\bm{u}_{h,D}^m,q_D)+\sum_{e\in \mathcal{F}_{h,\Gamma}}\frac{1}{\delta_p}(p_{D,h}^m,q_D)_e&=(f_D,q_D)_{\Omega_D}+\sum_{e\in \mathcal{F}_{h,\Gamma}}\frac{1}{\delta_p}(g_{D,h},q_D)_e\quad \forall q_D\in P_h^D\label{eq:Darcy-robink2}
\end{align}
and find $(\underline{\sigma}_h^m,\bm{u}_{S,h}^m)\in \Sigma_h^S\times U_h^S$ such that
\begin{align}
&((2\mu)^{-1}\mathcal{A}\underline{\sigma}_h^m,\underline{w})_{\Omega_S} = a_S(\underline{w},\bm{u}_{S,h}^m)\quad \forall\underline{w}\in \Sigma_h^S,\label{eq:Stokes-robink1}\\
&a_S(\underline{\sigma}_h^m,\bm{v}_S)+\sum_{e\in \mathcal{F}_{h,S}^0}\gamma h_e^{-1}(\jump{\bm{u}_{h,S}^m},\jump{\bm{v}_S})_e+\sum_{e\in \mathcal{F}_{h,\Gamma}}\delta_f(\bm{u}_{S,h}^m\cdot\bm{n}_S,\bm{v}_S\cdot\bm{n}_S)_{e}\nonumber\\
&\;+\frac{1}{G}\sum_{e\in \mathcal{F}_{h,\Gamma}}(\bm{u}_{S,h}^m\cdot\bm{t}_S,\bm{v}_S\cdot\bm{t}_S)_e=(\bm{f}_S,\bm{v}_S)_{\Omega_S}-\sum_{e\in \mathcal{F}_{h,\Gamma}}(g_{S,h},\bm{v}_S\cdot\bm{n}_S)_{e}\quad \forall \bm{v}_S\in U_h^S.\label{eq:Stokes-robink2}
\end{align}

\item[Step 3.] Update $g_{S,h}^{m+1}$ and $g_{D,h}^{m+1}$ in the following way:
\begin{align*}
g_{S,h}^{m+1}&=p_{D,h}^m(1+\frac{\delta_f}{\delta_p})-g_{D,h}^m \frac{\delta_f}{\delta_p},\\
g_{D,h}^{m+1} &= g_{S,h}^{m}+(\delta_f+\delta_p)\bm{u}_{S,h}^{m}\cdot\bm{n}_S.
\end{align*}

\end{itemize}

%If we also update $g_{S,h}^{k+1}$ in $g_{D,k}$ it will converge fast for $\mu$ very small (such as 0.001).

\begin{lemma}\label{lemma:graduSh}
There exists a positive constant $C$ independent of the mesh size such that the following estimate holds
\begin{align*}
\|\varepsilon_h (\bm{u}_{S,h})\|_{0,\Omega_S}\leq C \Big(\sup_{\underline{\psi}\in \Sigma_h^S}\frac{a_S(\underline{\psi},\bm{u}_{S,h})}{\|\underline{\psi}\|_{0,\Omega_S}}+\Big(\sum_{e\in \mathcal{F}_{h,S}^0}h_e^{-1}\|\jump{\bm{u}_{S,h}}\|_{0,e}^2\Big)^{1/2}\Big).
\end{align*}

\end{lemma}

\begin{proof}

We define $\underline{\psi}$ such that $\underline{\psi}=\varepsilon_h(\bm{u}_{S,h})\in \Sigma_h^S$, then it holds
\begin{equation}
\begin{split}
\|\varepsilon_h(\bm{u}_{S,h})\|_{0,\Omega_S}^2&=(\varepsilon_h(\bm{u}_{S,h}),\underline{\psi})_{\Omega_S}\\
&=(\varepsilon_h(\bm{u}_{S,h}),\Pi_h\underline{\psi})_{\Omega_S}-\sum_{e\in \mathcal{F}_{h,S}^0}(\avg{\Pi_h\underline{\psi}\bm{n}},\jump{\bm{u}_{S,h}})_e+\sum_{e\in \mathcal{F}_{h,S}^0}(\avg{\Pi_h\underline{\psi}\bm{n}},\jump{\bm{u}_{S,h}})_e\\
&=a_S(\Pi_h \underline{\psi},\bm{u}_{S,h})+\sum_{e\in \mathcal{F}_{h,S}^0}(\avg{\Pi_h\underline{\psi}\bm{n}},\jump{\bm{u}_{S,h}})_e.
\end{split}
\label{eq:uSh}
\end{equation}
The first term can be bounded by using the boundedness of $\Pi_h$
\begin{align*}
a_S(\Pi_h \underline{\psi},\bm{u}_{S,h})&= \frac{a_S(\Pi_h \underline{\psi},\bm{u}_{S,h})}{\|\Pi_h\underline{\psi}\|_{0,\Omega_S}}\|\Pi_h\underline{\psi}\|_{0,\Omega_S}\\
&\leq C \frac{a_S(\Pi_h \underline{\psi},\bm{u}_{S,h})}{\|\Pi_h\underline{\psi}\|_{0,\Omega_S}}\|\varepsilon_h(\bm{u}_{S,h})\|_{0,\Omega_S}\\
&\leq C \frac{a_S(\Pi_h \underline{\psi},\bm{u}_{S,h})}{\|\underline{\psi}\|_{0,\Omega_S}}\|\varepsilon_h(\bm{u}_{S,h})\|_{0,\Omega_S}.
\end{align*}
We can estimate the second term on the right hand side of \eqref{eq:uSh} via an application of the trace inequality
\begin{align*}
\sum_{e\in \mathcal{F}_{h,S}^0}(\avg{\Pi_h\underline{\psi}\bm{n}},\jump{\bm{u}_{S,h}})_e&\leq \epsilon\|\Pi_h\underline{\psi}\|_{0,\Omega_S}^2+\frac{1}{\epsilon}\sum_{e\in \mathcal{F}_{h,S}^0}h_e^{-1}\|\jump{\bm{u}_{S,h}}\|_{0,e}^2\\
&\leq \epsilon\|\varepsilon_h(\bm{u}_{S,h})\|_{0,\Omega_S}^2+\frac{1}{\epsilon}\sum_{e\in \mathcal{F}_{h,S}^0}h_e^{-1}\|\jump{\bm{u}_{S,h}}\|_{0,e}^2.
\end{align*}
Then taking $\epsilon$ small enough and using Young's inequality imply the desired estimate.

\end{proof}

The rest of this section is devoted to the proof for the convergence of Algorithm DDM. To simplify the notation, we let
$\eta_{S}^m=g_{S,h}-g_{S,h}^m$, $\eta_{D}^m=g_{D,h}-g_{D,h}^m$, $e_{\sigma}^m =\underline{\sigma}_h-\underline{\sigma}_h^m$, $e_{S}^m=\bm{u}_{S,h}-\bm{u}_{S,h}^m$, $e_D^m=\bm{u}_{D,h}-\bm{u}_{D,h}^m$ and $\varepsilon_D^m = p_{D,h}-p_{D,h}^m$.

%follow \cite{Chendd11,Sun21}
\begin{lemma}

The following identities are satisfied
\begin{equation}
\begin{split}
\|\eta_S^{m+1}\|_{0,\Gamma}^2
=\|\varepsilon_D^m\|_{0,\Gamma}^2(1-(\frac{\delta_f}{\delta_p})^2)-2(1+\frac{\delta_f}{\delta_p})\delta_f \|K^{-\frac{1}{2}} e_D^m\|_{0,\Omega_D}^2+(\frac{\delta_f}{\delta_p})^2\|\eta_D^m\|_{0,\Gamma}^2
\end{split}
\label{eq:etaSk}
\end{equation}
and
\begin{equation}
\begin{split}
\|\eta_D^{m+1}\|_{0,\Gamma}^2
&=\|\eta_S^m\|_{0,\Gamma}^2-2(\delta_f+\delta_p)\|(2\mu)^{-\frac{1}{2}}\mathcal{A}e_\sigma^m\|_{0,\Omega_S}^2-2(\delta_f+\delta_p)\sum_{e\in \mathcal{F}_{h,S}^0}\gamma h_e^{-1}\|\jump{e_S^m}\|_{0,e}^2\\
&\;-\frac{2(\delta_f+\delta_p)}{G}\|e_S^m\cdot\bm{t}_S\|_{0,\Gamma}^2+(\delta_p^2-\delta_f^2)\|e_{S}^m\cdot\bm{n}_S\|_{0,\Gamma}^2.
\end{split}
\label{eq:etaD}
\end{equation}

\end{lemma}

\begin{proof}

Subtracting \eqref{eq:Darcy-robink1}-\eqref{eq:Darcy-robink2} from \eqref{eq:Darcy-robin1}-\eqref{eq:Darcy-robin2}, and \eqref{eq:Stokes-robink1}-\eqref{eq:Stokes-robink2} from \eqref{eq:Stokes-robin1}-\eqref{eq:Stokes-robin2} yields
\begin{align}
(K^{-1} e_{D}^m,\bm{v}_D)_{\Omega_D}-b_{D}^*(\varepsilon_D^m,\bm{v}_D)&=0\quad \forall \bm{v}_D\in U_h^D,\label{eq:errorkD1}\\
b_{D}(e_{D}^m,q_D)+\sum_{e\in \mathcal{F}_{h,\Gamma}}\frac{1}{\delta_p}(\varepsilon_D^m,q_D)_{e}&=\sum_{e\in \mathcal{F}_{h,\Gamma}}\frac{1}{\delta_p}(\eta_D^m,q_D)_{e}\quad \forall q_D\in P_h^D\label{eq:errorkD2}
\end{align}
and
\begin{align}
&((2\mu)^{-1}\mathcal{A} e_{\sigma}^m,\underline{w})_{\Omega_S}=a_S(\underline{w},e_{S}^m)\quad \forall \underline{w}\in \Sigma_h^S,\label{eq:sigmak}\\
&a_S(e_\sigma^m,\bm{v}_S)+\sum_{e\in \mathcal{F}_{h,S}^0}\gamma h_e^{-1} (\jump{e_S^m},\jump{\bm{v}_S})_e+\sum_{e\in \mathcal{F}_{h,\Gamma}}\delta_f (e_S^m\cdot\bm{n}_S,\bm{v}_S\cdot\bm{n}_S)_{e}\nonumber\\
&\;+\sum_{e\in \mathcal{F}_{h,\Gamma}}\frac{1}{G}(e_S^m\cdot\bm{t}_S,\bm{v}_S\cdot\bm{t}_S)_{e}
=-\sum_{e\in \mathcal{F}_{h,\Gamma}}(\eta_S^m,\bm{v}_S\cdot\bm{n}_S)_{e}\quad \forall \bm{v}_S\in U_h^S\label{eq:usk}.
\end{align}
Taking $\bm{v}_D=e_D^m,q_D=\varepsilon_D^m,\bm{v}_S=e_S^m,\underline{w}=e_\sigma^m$ in the above equations and summing up the resulting equations, we can obtain
\begin{align}
&\|K^{-\frac{1}{2}} e_D^m\|_{0,\Omega_D}^2+\sum_{e\in \mathcal{F}_{h,\Gamma}}\frac{1}{\delta_p}\|\varepsilon_D^m\|_{0,e}^2=\sum_{e\in \mathcal{F}_{h,\Gamma}}\frac{1}{\delta_p}(\eta_D^m,\varepsilon_D^m)_{e},\label{eq:eDk1}\\
&\|(2\mu)^{-\frac{1}{2}}\mathcal{A}e_\sigma^m\|_{0,\Omega_S}^2+\sum_{e\in \mathcal{F}_{h,S}^0}\gamma h_e^{-1}\|\jump{e_S^m}\|_{0,e}^2+\sum_{e\in \mathcal{F}_{h,\Gamma}}\delta_f \|e_S^m\cdot\bm{n}_S\|_{0,e}^2\nonumber\\
&\;+\sum_{e\in \mathcal{F}_{h,\Gamma}}\frac{1}{G}\|e_S^m\cdot\bm{t}_S\|_{0,e}^2=-\sum_{e\in \mathcal{F}_{h,\Gamma}}(\eta_S^m,e_S^m\cdot\bm{n}_{S})_{e}.\label{eq:eDk2}
\end{align}
The following identities are satisfied on the interface $\Gamma$
\begin{align*}
\eta_S^{m+1}&=\varepsilon_D^m(1+\frac{\delta_f}{\delta_p})-\eta_D^m\frac{\delta_f}{\delta_p} ,\\
\eta_D^{m+1}&=\eta_S^m+(\delta_f+\delta_p)e_S^m\cdot\bm{n}_S.
\end{align*}
Therefore, we can infer from \eqref{eq:eDk1} and \eqref{eq:eDk2} that
\begin{align*}
\|\eta_S^{m+1}\|_{0,\Gamma}^2&=\|\varepsilon_D^m(1+\frac{\delta_f}{\delta_p})-\eta_D^m\frac{\delta_f}{\delta_p}\|_{0,\Gamma}^2\\
&=\|\varepsilon_D^m\|_{0,\Gamma}^2(1-(\frac{\delta_f}{\delta_p})^2)-2(1+\frac{\delta_f}{\delta_p})\delta_f \|K^{-\frac{1}{2}} e_D^m\|_{0,\Omega_D}^2+(\frac{\delta_f}{\delta_p})^2\|\eta_D^m\|_{0,\Gamma}^2
\end{align*}
and
\begin{align*}
\|\eta_D^{m+1}\|_{0,\Gamma}^2&=\|\eta_S^m+(\delta_f+\delta_p)e_S^m\cdot\bm{n}_S\|_{0,\Gamma}^2\\
&=\|\eta_S^m\|_{0,\Gamma}^2-2(\delta_f+\delta_p)\|(2\mu)^{-\frac{1}{2}}\mathcal{A}e_\sigma^m\|_{0,\Omega_S}^2-2(\delta_f+\delta_p)\sum_{e\in \mathcal{F}_{h,S}^0}\gamma h_e^{-1}\|\jump{e_S^m}\|_{0,e}^2\\
&\;-2(\delta_f+\delta_p)\delta_f \|e_S^m\cdot\bm{n}_S\|_{0,\Gamma}^2-\frac{2(\delta_f+\delta_p)}{G}\|e_S^m\cdot\bm{t}_S\|_{0,\Gamma}^2
+(\delta_f+\delta_p)^2\|e_{S}^m\cdot\bm{n}_S\|_{0,\Gamma}^2\\
&=\|\eta_S^m\|_{0,\Gamma}^2-2(\delta_f+\delta_p)\|(2\mu)^{-\frac{1}{2}}\mathcal{A}e_\sigma^m\|_{0,\Omega_S}^2-2(\delta_f+\delta_p)\sum_{e\in \mathcal{F}_{h,S}^0}\gamma h_e^{-1}\|\jump{e_S^m}\|_{0,e}^2\\
&\;-\frac{2(\delta_f+\delta_p)}{G}\|e_S^m\cdot\bm{t}_S\|_{0,\Gamma}^2+(\delta_p^2-\delta_f^2)\|e_{S}^k\cdot\bm{n}_S\|_{0,\Gamma}^2,
\end{align*}
which implies the desired estimates.

\end{proof}

\begin{theorem}\label{thm:convergence-DDM}
If $\delta_p=\delta_f=\delta$, then the solution of Algorithm DDM converges to the solution of the Stokes-Darcy system \eqref{eq:discrete1}-\eqref{eq:discrete3}. Moreover, the following estimate holds
\begin{align*}
&\|\mu^{-\frac{1}{2}}\mathcal{A} e_{\sigma}^m\|_{0,\Omega_S}^2+\|K^{-\frac{1}{2}} e_{D}^m\|_{0,\Omega_D}^2+\|e_S^m\|_{1,h}^2+\|\eta_D^{m+1}\|_{0,\Gamma}^2+\|\eta_{S}^{m+1}\|_{0,\Gamma}^2\\
&\leq
 C(1-Ch(h+C\delta K)^{-1})^{[\frac{m}{2}]}\Big(\|\eta_D^{0}\|_{0,\Gamma}^2+\|\eta_S^{0}\|_{0,\Gamma}^2\Big).
\end{align*}
If $0<\delta_p-\delta_f\leq \frac{\min\{\sqrt{2}\mu,1/\gamma\}}{C_S^2}$ and $\frac{1}{\delta_f}-\frac{1}{\delta_p}\leq \frac{2}{C_p^2K}$, then the solution of Algorithm DDM converges to the solution of the Stokes-Darcy system \eqref{eq:discrete1}-\eqref{eq:discrete3}. In addition, the following estimate holds
\begin{align*}
&\|\mu^{-\frac{1}{2}}\mathcal{A} e_{\sigma}^m\|_{0,\Omega_S}^2+\|K^{-\frac{1}{2}} e_{D}^m\|_{0,\Omega_D}^2+\|e_S^m\|_{1,h}^2+\|\eta_D^{m+1}\|_{0,\Gamma}^2+\|\eta_{S}^{m+1}\|_{0,\Gamma}^2\\
&\leq C
(\frac{\delta_f}{\delta_p})^{m}\Big(\|\eta_D^{0}\|_{0,\Gamma}^2+\|\eta_S^{0}\|_{0,\Gamma}^2\Big).
\end{align*}

\end{theorem}

\begin{proof}

Let $\mathcal{N}_{D,h}$ denote the set of nodes corresponding to the degrees of freedom specified in (SD1)-(SD2) in $\Omega_D$ and let $\mathcal{N}_{D,\Gamma}=\mathcal{N}_{D,h}|{\Gamma}$. We define an extension operator $E_{D,h}$, which satisfies
\begin{align*}
E_{D,h}\eta_D^m(P)=
\begin{cases}
\eta_D^m\quad\; \mbox{if}\;P\in \mathcal{N}_{D,\Gamma},\\
0\qquad \mbox{if}\;P\in \mathcal{N}_{D,h}\backslash\mathcal{N}_{D,\Gamma}.
\end{cases}
\end{align*}
Note that $E_{D,h}\eta_D^m\in P_h^D$ and an application of scaling arguments implies that
\begin{align}
\|E_{D,h}\eta_D^m\|_{0,\Omega_D}^2\leq C h\|\eta_D^m\|_{0,\Gamma}^2.\label{eq:EDh}
\end{align}
Then setting $q_D=E_{D,h} \eta_D^m$ and $\bm{v}_D=e_D^m$ in \eqref{eq:errorkD1}-\eqref{eq:errorkD2} and summing up the resulting equations, we can get
\begin{align}
\frac{1}{\delta_p}\|\eta_D^m\|_{0,\Gamma}^2=\|K^{-\frac{1}{2}} e_D^m\|_{0,\Omega_D}^2
-b_{D}^*(\varepsilon_D^m,e_D^m)+b_{D}(e_D^m,E_{D,h}\eta_D^m)+\frac{1}{\delta_p}(\varepsilon_D^m,\eta_D^m)_\Gamma.\label{eq:etaDk}
\end{align}
The inf-sup condition \eqref{eq:inf-sup} and \eqref{eq:errorkD1} imply that
\begin{align}
\|\varepsilon_D^m\|_{Z}\leq C_pK^{-\frac{1}{2}}\|K^{-\frac{1}{2}}(\bm{u}_{D,h}-\bm{u}_{D,h}^m)\|_{0,\Omega_D}.\label{eq:varepsilonDk}
\end{align}
The trace inequality implies that
\begin{align}
\|\varepsilon_D^m\|_{0,\Gamma}\leq C \|\varepsilon_D^m\|_Z.\label{eq:varepsilonDk2}
\end{align}
An application of the inverse inequality and \eqref{eq:EDh} reveals that
\begin{align*}
\|E_{D,h}\eta_D^m\|_{Z}\leq C h^{-1}\|E_{D,h}\eta_{D}^m\|_{0,\Omega_D}\leq C h^{-\frac{1}{2}}\|\eta_{D}^m\|_{0,\Gamma}.
\end{align*}
Thereby, we have from \eqref{eq:etaDk} and the Cauchy-Schwarz inequality that
\begin{equation}
\begin{split}
\frac{1}{\delta_p}\|\eta_D^m\|_{0,\Gamma}^2&\leq \|K^{-\frac{1}{2}} e_D^m\|_{0,\Omega_D}^2+C\|\varepsilon_D^m\|_Z\|e_D^m\|_{0,\Omega_D}+C\|e_D^m\|_0\|E_{D,h}\eta_D^m\|_{Z}+
\frac{1}{\delta_p}\|\varepsilon_D^m\|_{0,\Gamma}\|\eta_D^m\|_{0,\Gamma}\\
 &\leq C(1+Ch^{-1}\delta_pK)\|K^{-\frac{1}{2}}e_D^m\|_{0,\Omega_D}^2.
 \end{split}
 \label{eq:etaDkeDk}
\end{equation}
If $\delta_p=\delta_f=\delta$, it follows from \eqref{eq:etaSk} and \eqref{eq:etaD} that
\begin{align}
\|\eta_S^{m+1}\|_{0,\Gamma}^2&=-4 \delta\|K^{-\frac{1}{2}} e_D^m\|_{0,\Omega_D}^2+\|\eta_D^m\|_{0,\Gamma}^2\label{eq:delta-equal1},\\
\|\eta_D^{m+1}\|_{0,\Gamma}^2&=\|\eta_S^m\|_{0,\Gamma}^2-4\delta\|(2\mu)^{-\frac{1}{2}}e_\sigma^m\|_{0,\Omega_S}^2
-\frac{4\delta}{G}\|e_S^m\cdot\bm{t}_S\|_{0,\Gamma}^2-4\delta\sum_{e\in \mathcal{F}_{h,S}^0}\gamma h_e^{-1}\|\jump{e_S^m}\|_{0,e}^2.\label{eq:delta-equal2}
\end{align}
We can infer from \eqref{eq:delta-equal2} that
\begin{align}
\|\eta_D^{m+1}\|_{0,\Gamma}\leq \|\eta_S^m\|_{0,\Gamma}.\label{eq:etaDetaS}
\end{align}
Now it remains to estimate \eqref{eq:delta-equal1}.

Combining \eqref{eq:etaDkeDk} with \eqref{eq:delta-equal1} yields
\begin{align*}
\|\eta_S^{m+1}\|_{0,\Gamma}^2\leq \|\eta_D^m\|_{0,\Gamma}^2(1-4 (1+Ch^{-1}\delta K)^{-1}).
\end{align*}
which can be combined with \eqref{eq:etaDetaS} implies
\begin{align*}
\|\eta_D^{m+1}\|_{0,\Gamma}^2&\leq (1-Ch (h+\delta K)^{-1})\|\eta_D^{m-1}\|_{0,\Gamma}^2,\\
\|\eta_{S}^{m+1}\|_{0,\Gamma}^2&\leq (1-Ch (h+\delta K)^{-1})\|\eta_S^{m-1}\|_{0,\Gamma}^2.
\end{align*}
Thus
\begin{align*}
\|\eta_D^{m+1}\|_{0,\Gamma}^2+\|\eta_{S}^{m+1}\|_{0,\Gamma}^2&\leq
 (1-Ch(h+\delta K)^{-1})^{[\frac{m}{2}]}\Big(\|\eta_D^{0}\|_{0,\Gamma}^2+\|\eta_S^{0}\|_{0,\Gamma}^2\Big).
\end{align*}
Then taking $\bm{v}_D=e_D^m$ and $q_D=\varepsilon_D^m$ in \eqref{eq:errorkD1}-\eqref{eq:errorkD2} and summing up the resulting equations yield
\begin{align*}
\|K^{-\frac{1}{2}} e_{D}^m\|_{0,\Omega_D}^2+\frac{1}{\delta_p}\|\varepsilon_D^m\|_{0,\Gamma}^2=\frac{1}{\delta_p}(\eta_D^m,\varepsilon_D^m)_{\Gamma}.
\end{align*}
Thus
\begin{align*}
\|K^{-\frac{1}{2}} e_{D}^m\|_{0,\Omega_D}^2+\frac{1}{\delta_p}\|\varepsilon_D^m\|_{0,\Gamma}^2\leq C\|\eta_D^m\|_{0,\Gamma}^2.
\end{align*}
Taking $\underline{w}=e_{\sigma}^m, \bm{v}_S=e_S^m$ in \eqref{eq:sigmak}-\eqref{eq:usk} and summing up the resulting equations, we can obtain
\begin{align*}
\|(2\mu)^{-\frac{1}{2}}\mathcal{A} e_{\sigma}^m\|_{0,\Omega_S}^2+\sum_{e\in \mathcal{F}_{h,S}^0}\gamma h_e^{-1} ||\jump{e_S^m}\|_{0,e}^2
+\delta_f \|e_S^m\cdot\bm{n}_S\|_{0,\Gamma}^2+
\frac{1}{G}\|e_S^m\cdot\bm{t}_S\|_{0,\Gamma}^2\leq \|\eta_S^m\|_{0,\Gamma}\|\bm{e}_S^m\cdot\bm{n}_S\|_{0,\Gamma}.
\end{align*}
The discrete Korn's inequality, the trace inequality, \eqref{eq:Stokes-robink1} and Lemma~\ref{lemma:graduSh} yield
\begin{align}
\|\bm{e}_S^m\cdot\bm{n}_S\|_{0,\Gamma}\leq \|e_S^m\|_{1,h}\leq C_S\Big(\|\mu^{-1}\mathcal{A}e_\sigma^m\|_{0,\Omega_S}+\Big(\sum_{e\in \mathcal{F}_{h,S}^0}h_e^{-1}\|\jump{e_S^m}\|_{0,e}^2\Big)^{\frac{1}{2}}\Big).\label{eq:esk}
\end{align}
Thus, Young's inequality implies
\begin{align*}
\|\mu^{-\frac{1}{2}}\mathcal{A} e_{\sigma}^m\|_{0,\Omega_S}^2+\sum_{e\in \mathcal{F}_{h,S}^0}\gamma h_e^{-1} ||\jump{e_S^m}\|_{0,e}^2
+\delta_f \|e_S^m\cdot\bm{n}_S\|_{0,\Gamma}^2+
\frac{1}{G}\|e_S^m\cdot\bm{t}_S\|_{0,\Gamma}^2\leq C\|\eta_S^m\|_{0,\Gamma}^2.
\end{align*}
Combining the preceding arguments, we obtain the following convergence for $\delta_f=\delta_p=\delta$
\begin{equation}
\begin{split}
&\|\mu^{-\frac{1}{2}}\mathcal{A} e_{\sigma}^m\|_{0,\Omega_S}^2+\|K^{-\frac{1}{2}} e_{D}^m\|_{0,\Omega_D}^2+\|e_S^m\|_{1,h}^2+\|\eta_D^{m+1}\|_{0,\Gamma}^2+\|\eta_{S}^{m+1}\|_{0,\Gamma}^2\\
&\leq
 C(1-Ch(h+C\delta K)^{-1})^{[\frac{m}{2}]}\Big(\|\eta_D^{0}\|_{0,\Gamma}^2+\|\eta_S^{0}\|_{0,\Gamma}^2\Big).
\end{split}
\label{eq:etaS}
\end{equation}
If $\delta_f<\delta_p$ and $\delta_p-\delta_f\leq \frac{\min\{\sqrt{2}\mu,1/\gamma\}}{C_S^2} $, we have from \eqref{eq:esk} that
\begin{align*}
(\delta_p^2-\delta_f^2)\|e_{S}^m\cdot\bm{n}_S\|_{0,\Gamma}^2\leq 2(\delta_p+\delta_f)\Big(\|(2\mu)^{-\frac{1}{2}}\mathcal{A}e_\sigma^m\|_{0,\Omega_S}^2
+\sum_{e\in \mathcal{F}_{h,S}^0}\gamma h_e^{-1} ||\jump{e_S^m}\|_{0,e}^2\Big).
\end{align*}
Thereby, it follows from \eqref{eq:etaD} that
\begin{align*}
\|\eta_D^{m+1}\|_{0,\Gamma}\leq \|\eta_S^m\|_{0,\Gamma}.
\end{align*}
On the other hand, we can deduce from \eqref{eq:varepsilonDk}-\eqref{eq:varepsilonDk2} that if $\frac{1}{\delta_f}-\frac{1}{\delta_p}\leq \frac{2K}{C_p^2}$, then we can obtain
\begin{align*}
&\|\varepsilon_D^m\|_{0,\Gamma}^2(1-(\frac{\delta_f}{\delta_p})^2)-2(1+\frac{\delta_f}{\delta_p})\delta_f \|K^{-\frac{1}{2}} e_D^m\|_{0,\Omega_D}^2\\
&\leq\delta_f(1+\frac{\delta_f}{\delta_p})\Big((\frac{1}{\delta_f}-\frac{1}{\delta_p})C_p^2K^{-1}-2\Big)\|K^{-\frac{1}{2}} e_D^m\|_{0,\Omega_D}^2\leq 0.
\end{align*}
Thus, an appeal to \eqref{eq:etaSk} leads to
\begin{align*}
\|\eta_S^{m+1}\|_{0,\Gamma}^2\leq (\frac{\delta_f}{\delta_p})^2\|\eta_D^m\|_{0,\Gamma}^2.
\end{align*}
Then we can get
\begin{align*}
\|\eta_D^{m+1}\|_{0,\Gamma}^2+\|\eta_{S}^{m+1}\|_{0,\Gamma}^2&\leq
C (\frac{\delta_f}{\delta_p})^{m}\Big(\|\eta_D^{0}\|_{0,\Gamma}^2+\|\eta_S^{0}\|_{0,\Gamma}^2\Big).
\end{align*}
Therefore, proceeding similarly to \eqref{eq:etaS}, we can obtain
\begin{align*}
&\|\mu^{-\frac{1}{2}}\mathcal{A} e_{\sigma}^m\|_{0,\Omega_S}^2+\|K^{-\frac{1}{2}} e_{D}^m\|_{0,\Omega_D}^2+\|e_S^m\|_{1,h}^2+\|\eta_D^{m+1}\|_{0,\Gamma}^2+\|\eta_{S}^{m+1}\|_{0,\Gamma}^2\\
&\leq C
(\frac{\delta_f}{\delta_p})^{m}\Big(\|\eta_D^{0}\|_{0,\Gamma}^2+\|\eta_S^{0}\|_{0,\Gamma}^2\Big).
\end{align*}
Therefore, the proof is completed.

%Summing up $k$ yields
%\begin{align*}
%&\|\eta_S^{N+1}\|_{0,\Gamma}^2+\|\eta_D^{N+1}\|_{0,\Gamma}^2+4\delta \sum_{k=0}^N(\|K^{-1/2} e_D^k\|_{0,\Omega_D}^2
%+|\mu^{-1/2}e_\sigma^k\|_{0,\Omega_S}^2+4\sum_{e\in \mathcal{F}_{h,S}^0}\gamma h_e^{-1}\|\jump{e_S^k}\|_{0,e}^2)\\
%&=\|\eta_S^1\|_{0,\Gamma}^2+\|\eta_D^1\|_{0,\Gamma}^2+4\delta (\|K^{-1/2} e_D^0\|_{0,\Omega_D}^2
%+|\mu^{-1/2}e_\sigma^0\|_{0,\Omega_S}^2+4\sum_{e\in \mathcal{F}_{h,S}^0}\gamma h_e^{-1}\|\jump{e_S^0}\|_{0,e}^2)
%\end{align*}
%which implies $e_\sigma^k$ and $e_D^k$ approaches zero in $H^1(\Omega_S)$ and $L^2(\Omega_D)$ as $k$ goes to $\infty$.
%
%
%We can infer from \eqref{eq:sigmak} that
%\begin{align*}
%\|2\mu^{-1/2}e_\sigma^k\|_0\leq \|e_{S}^k\|_{1,h}
%\end{align*}
%
%By inf-sup condition for $b_{D}$ and the discrete Poincar\'{e} inequality, we get the convergence of $\varepsilon_D^k$.
%

\end{proof}

\section{Numerical experiments}\label{sec:numerical}

In this section, several numerical experiments will be presented to show the performance of the proposed scheme. First, we test the convergence of the proposed scheme (cf. \eqref{eq:discrete1}-\eqref{eq:discrete3}) for different polynomial orders. In addition, the robustness of the scheme with respect to different values of $\mu$ is demonstrated. Then, we show the convergence of the Algorithm DDM with respect to different values of $\delta_p$ and $\delta_f$. Note that we use the uniform criss-cross meshes in the following examples and similar performance can be observed for other types of triangular meshes. In the following examples, we set $\gamma=1$. The stopping criterion for the iteration of the
algorithm is selected as a fixed tolerance of $10^{-6}$ for the difference between two successive iterative velocities in
$L^2$-norm, i.e.,
\begin{align*}
\|\bm{u}_{S,h}^{m+1}-\bm{u}_{S,h}^m\|_{0,\Omega_S}+\|\bm{u}_{D,h}^{m+1}-\bm{u}_{D,h}^m\|_{0,\Omega_D}\leq 10^{-6}.
\end{align*}
%
%for test 3, the second case works for $\mu=0.001$ ($\delta_p = \mu,
%\delta_f = 2\mu$ or $\delta_p = 2\mu,
%\delta_f = \mu$). The first works for $\nu=1$.
%The error

\subsection{Example 1}\label{ex1}

In this example, we take $\Omega_D=(0,1)^2$ and $\Omega_S=(0,1)\times (1,2)$. The exact solution is given by
\begin{align*}
 \bm{u}_S=
 \begin{cases}
-\cos(\pi y/2)^2\sin(\pi x/2)\\
\cos(\pi x/2)(\sin(\pi y)+\pi y)/4
\end{cases},\quad p_D=-\pi\cos(\pi x/2)y/4,
\end{align*}
which does not satisfy the interface conditions, in this respect, the discrete formulation shall be adapted to account for this situation. The convergence history for various values of $\mu$ for polynomial orders $k=1,2,3$ are displayed in Tables~\ref{ex1:table1}-\ref{ex1:table2}. We can observe that optimal convergence rates for all the variables measured in $L^2$-error can be obtained. In addition, the accuracy for $L^2$-error of fluid velocity is slightly influenced by the values of $\mu$, which demonstrates the robustness of our method with respect to $\mu$.

\begin{table}[t]
\begin{center}
{\footnotesize
\begin{tabular}{cc||c c|c c|c c|c c}
\hline
 &Mesh & \multicolumn{2}{|c|}{$\|\bm{u}_S-\bm{u}_{S,h}\|_{0,\Omega_S}$} & \multicolumn{2}{|c|}{$\|\underline{\sigma}-\underline{\sigma}_h\|_{0,\Omega_S}$} & \multicolumn{2}{|c}{$\|\bm{u}_D-\bm{u}_{D,h}\|_{0,\Omega_D}$} &  \multicolumn{2}{|c}{$\|p_D-p_{D,h}\|_{0,\Omega_D}$} \\
\hline
$k$ & $h^{-1}$  & Error & Order & Error& Order& Error & Order & Error &Order\\
\hline
1& 2   & 2.42e-02 &   N/A   &5.61e-01 &  N/A &1.69e-02 &   N/A & 6.90e-03 &N/A\\
&  4  & 6.70e-03 &  1.85  &3.02e-01 & 0.89 &4.10e-03 & 2.05  &1.70e-03 &2.01 \\
&  8  & 1.70e-03 &  1.95  &1.56e-01 & 0.95 &9.45e-04& 2.03 & 4.29e-04 & 2.00\\
&  16   &4.44e-04 & 1.97  &7.95e-02 & 0.98 &2.46e-04 &2.02 & 1.07e-04 & 2.00\\
&  32  & 1.12e-04 & 1.99  &4.01e-02 & 0.99 &6.09e-05 &2.01& 2.68e-05 &  2.00 \\
\hline
2& 2   & 2.80e-03 &   N/A   &8.71e-02 &  N/A &7.86e-04 &   N/A & 3.00e-04&N/A\\
&  4  & 3.91e-04 &  2.83  &2.35e-02 & 1.89 &8.64e-05 & 3.18 & 3.69e-05 &3.02\\
&  8  & 5.25e-05 &  2.90  &6.20e-03 & 1.92 &1.02e-05& 3.07&4.58e-06 &3.00\\
&  16   &6.79e-06 & 2.95  &1.60e-03 & 1.96 &1.27e-06 &3.02&5.73e-07 &3.00\\
&  32  & 8.64e-07 & 2.98  &4.04e-04 & 1.98 &1.58e-07 &3.00 &7.16e-08&3.00 \\
\hline
3& 2   &2.19e-04 &   N/A  &7.60e-03 &  N/A &4.96e-05 &   N/A &7.61e-06&N/A\\
&  4  & 1.42e-05 &  3.95  &1.00e-03 & 2.91 &2.33e-06 & 4.41 &4.18e-07 &4.18 \\
&  8  & 9.00e-07 &  3.98  &1.29e-04 & 2.96 &1.31e-07& 4.15&2.58e-08 &4.01\\
&  16  &5.66e-08 & 3.99 &1.64e-05 & 2.98 &7.80e-09 &4.06&1.61e-09&4.00\\
&  32  &3.55e-09 & 3.99  &2.06e-06 & 2.99&4.77e-010 &4.03&1.00e-10&3.99  \\
\hline
\end{tabular}}
\caption{Convergence history with $\mu=1$ for Example~\ref{ex1}.}
\label{ex1:table1}
\end{center}
\end{table}

\begin{table}[t]
\begin{center}
{\footnotesize
\begin{tabular}{cc||c c|c c|c c|c c}
\hline
 &Mesh & \multicolumn{2}{|c|}{$\|\bm{u}_S-\bm{u}_{S,h}\|_{0,\Omega_S}$} & \multicolumn{2}{|c|}{$\|\underline{\sigma}-\underline{\sigma}_h\|_{0,\Omega_S}$} & \multicolumn{2}{|c}{$\|\bm{u}_D-\bm{u}_{D,h}\|_{0,\Omega_D}$} &  \multicolumn{2}{|c}{$\|p_D-p_{D,h}\|_{0,\Omega_D}$} \\
\hline
$k$ & $h^{-1}$  & Error & Order & Error& Order& Error & Order & Error &Order\\
\hline
1& 2   & 4.99e-02 &   N/A   &1.33e-01 &  N/A &2.34e-02 &   N/A & 7.00e-03 &N/A\\
&  4  & 1.57e-02 &  1.67  &5.72e-02 & 1.21 &3.90e-03 & 2.58  &1.70e-03 &2.05 \\
&  8  & 3.60e-03 &  2.11  &2.83e-02 & 1.02 &9.49e-04& 2.04 & 4.25e-04 & 1.99\\
&  16   &8.79e-04 & 2.04  &1.42e-02 & 1.00 &2.37e-04 &2.00 & 1.06e-04 & 2.00\\
&  32  & 2.17e-04 & 2.02  &7.10e-03 & 1.00 &5.93e-05 &2.00& 2.65e-05 &  2.00 \\
\hline
2& 2   & 6.66e-03 &   N/A   &2.35e-02 &  N/A &1.30e-03 &   N/A & 3.06e-04&N/A\\
&  4  & 8.23e-04 &  2.99  &6.00e-03 & 1.97 &1.22e-04 & 3.36 & 3.70e-05 &3.04\\
&  8  & 1.15e-04 &  2.84  &1.50e-03 & 1.99 &1.29e-05& 3.24&4.57e-06 &3.01\\
&  16   &1.53e-05 & 2.91  &3.75e-04 & 1.99 &1.63e-06 &2.98&5.72e-07 &3.00\\
&  32  & 1.97e-06 & 2.96  &9.38e-05 & 2.00 &1.95e-07 &3.06 &7.15e-08&2.99 \\
\hline
3& 2   & 7.27e-04 &   N/A  &1.80e-03 &  N/A &7.91e-05 &   N/A &7.46e-06&N/A\\
&  4  & 3.37e-05 &  4.43  &2.23e-04 & 2.99 &2.77e-06 & 4.83 &4.03e-07 &4.21 \\
&  8  & 1.91e-06 &  4.14  &2.80e-05 & 2.99 &1.62e-07& 4.09&2.55e-08 &3.98\\
&  16  &1.13e-07 & 4.07  &5.49e-06 & 3.00 &1.06e-08 &3.93&1.61e-09&3.99\\
&  32  &6.87e-09 & 4.04  &4.37e-07 & 3.00 &6.25e-010 &4.08&1.00e-10&3.99  \\
\hline
\end{tabular}}
\caption{Convergence history with $\mu=10^{-4}$ for Example~\ref{ex1}.}
\label{ex1:table2}
\end{center}
\end{table}

\subsection{Example 2}\label{ex2}

In this example, we set $\Omega_D=(0,1)\times (0,0.5)$ and $\Omega_S=(0,1)\times (0.5,1)$. The exact solution is defined by
\begin{align*}
 \bm{u}_S=
 \begin{cases}
-\sin(\pi x)\text{exp}(y/2)/(2\pi^2)\\
 \cos(\pi x)\text{exp}(y/2)/\pi
\end{cases},\quad p_D=-2\cos(\pi x)\text{exp}(y/2)/\pi.
\end{align*}
Here we set $G = 2/(1+4\pi^2)$ and $\mu=1$. In addition, we set the polynomial order $k=1$. We aim to test the convergence of Algorithm DDM for different values of $\delta_f$ and $\delta_f$.
First, we let $\delta_f=\delta_p=\delta$ and $\delta=0.5,0.25, 0.1$. We can observe from Table~\ref{ex2:table1} that the convergence rate of the algorithm is $h$-dependent for the case of $\delta_f=\delta_p$; indeed, more iterations are needed for smaller meshsize. This is consistent with our theoretical results as the converge rate of Algorithm DDM is proved to depend on $1-Ch$. In addition, the solution converges faster when $\delta$ is smaller. Next, we set $\delta_f<\delta_p$. We let $\delta_p=1$ and $\delta_f=1/2\delta_p,1/4\delta_p$. We can see from Table~\ref{ex2:table2} that less iterations are required compared to the case of $\delta_p=\delta_f$. In addition, the convergence rate is $h$-independent, which is consistent with our analysis given in Theorem~\ref{thm:convergence-DDM}. In both cases, we can achieve the optimal convergence rates for $L^2$ errors of all the variables. Moreover, we also display the velocity errors for both the Stokes and Darcy regions with respect to different choices of $\delta_p$ and $\delta_f$ in Figure~\ref{ex2:fig}. It shows that Algorithm DDM converges for $\delta_f\leq \delta_p$ and it tends to converge faster for smaller ratio of $\frac{\delta_f}{\delta_p}$, which is consistent with our theory.
%Test the convergence of the method and the convergence of iterative method with respect to different choices of $\delta_f$ and $\delta_p$

%We set $\mu=1$ and $G=1$, the polynomial order $k=1$.

\begin{table}[t]
\begin{center}
{\footnotesize
\begin{tabular}{ccc||c c|c c|c c|c c}
\hline
 &Mesh & Iteration& \multicolumn{2}{|c|}{$\|\bm{u}_S-\bm{u}_{S,h}\|_{0,\Omega_S}$} & \multicolumn{2}{|c|}{$\|\underline{\sigma}-\underline{\sigma}_h\|_{0,\Omega_S}$} & \multicolumn{2}{|c}{$\|\bm{u}_D-\bm{u}_{D,h}\|_{0,\Omega_D}$} &  \multicolumn{2}{|c}{$\|p_D-p_{D,h}\|_{0,\Omega_D}$} \\
\hline
$\delta$ & $h^{-1}$ &N  & Error & Order & Error& Order& Error & Order & Error &Order\\
\hline
0.5& 4  &144  & 3.50e-03 &   N/A   &1.29e-01 &  N/A &2.07e-02 &   N/A &4.20e-03 &N/A\\
&  8  &  264  & 8.78e-04 &  1.99  &6.94e-02 & 0.90 &5.10e-03 & 2.01  &1.10e-03 &1.99 \\
&  16 &312 & 2.17e-04 &  2.01  &3.62e-02 & 0.94 &1.30e-03& 2.00 & 2.65e-04 & 1.99\\
\hline
0.25& 4  &52  & 3.50e-03 &   N/A   &1.29e-01 &  N/A &2.07e-02 &   N/A & 4.20e-03 &N/A\\
&  8  & 98  & 8.78e-04 &  1.99  &6.94e-02 & 0.90 &5.10e-03 & 2.01  &1.11e-03 &1.99 \\
&  16 & 180 & 2.17e-04 &  2.01  &3.62e-02 & 0.94 &1.30e-03& 2.00 & 2.65e-04 & 1.99\\
\hline
0.1& 4  &22 &3.50e-03 &   N/A   &1.29e-01 &  N/A &2.07e-02 &   N/A & 4.20e-03  &N/A\\
&  8  &  42  & 8.78e-04 &  1.99  &6.94e-02 & 0.90 &5.10e-03 & 2.01  &1.11e-03  &1.99 \\
&  16 & 76 & 2.17e-04 &  2.01  &3.62e-02 & 0.94 &1.30e-04& 2.00 & 2.65e-04  & 1.99\\
\hline
\end{tabular}}
\caption{The convergence of Algorithm DDM for $\delta_p=\delta_f=\delta$ for Example~\ref{ex2}.}
\label{ex2:table1}
\end{center}
\end{table}

\begin{table}[t]
\begin{center}
{\footnotesize
\begin{tabular}{cccc||c| c|c | c}
\hline
$\delta_p$ &$\delta_f$ &$h^{-1}$ & Iteration& $\|\bm{u}_S-\bm{u}_{S,h}\|_{0,\Omega_S}$ & $\|\underline{\sigma}-\underline{\sigma}_h\|_{0,\Omega_S}$ & $\|\bm{u}_D-\bm{u}_{D,h}\|_{0,\Omega_D}$ &  $\|p_D-p_{D,h}\|_{0,\Omega_D}$\\
\hline
1 &1/2 & 4 &28  & 3.50e-03    &1.29e-01  &2.07e-02  &4.20e-03\\
& &  8  & 30  & 8.78e-04   &6.94e-02 &5.10e-03   &1.10e-03 \\
& & 16 &30 & 2.17e-04  &3.62e-02  &1.30e-03 & 2.65e-04 \\
& & 32 &30 & 5.40e-05  &1.85e-02  &3.21e-04 & 6.63e-05 \\
\hline
1 &1/4 & 4 &16  & 3.50e-03    &1.29e-01  &2.07e-02  &4.20e-03\\
& &  8  & 16  & 8.78e-04   &6.94e-02 &5.10e-03   &1.10e-03 \\
& & 16 &16 & 2.17e-04  &3.62e-02  &1.30e-03 & 2.65e-04 \\
& & 32 &16 & 5.40e-05  &1.85e-02  &3.21e-04 & 6.63e-05 \\
\hline
\end{tabular}}
\caption{The convergence of Algorithm DDM for $\delta_f<\delta_p$ for Example~\ref{ex2}.}
\label{ex2:table2}
\end{center}
\end{table}

\begin{figure}[t]
\centering
\includegraphics[width=0.45\textwidth]{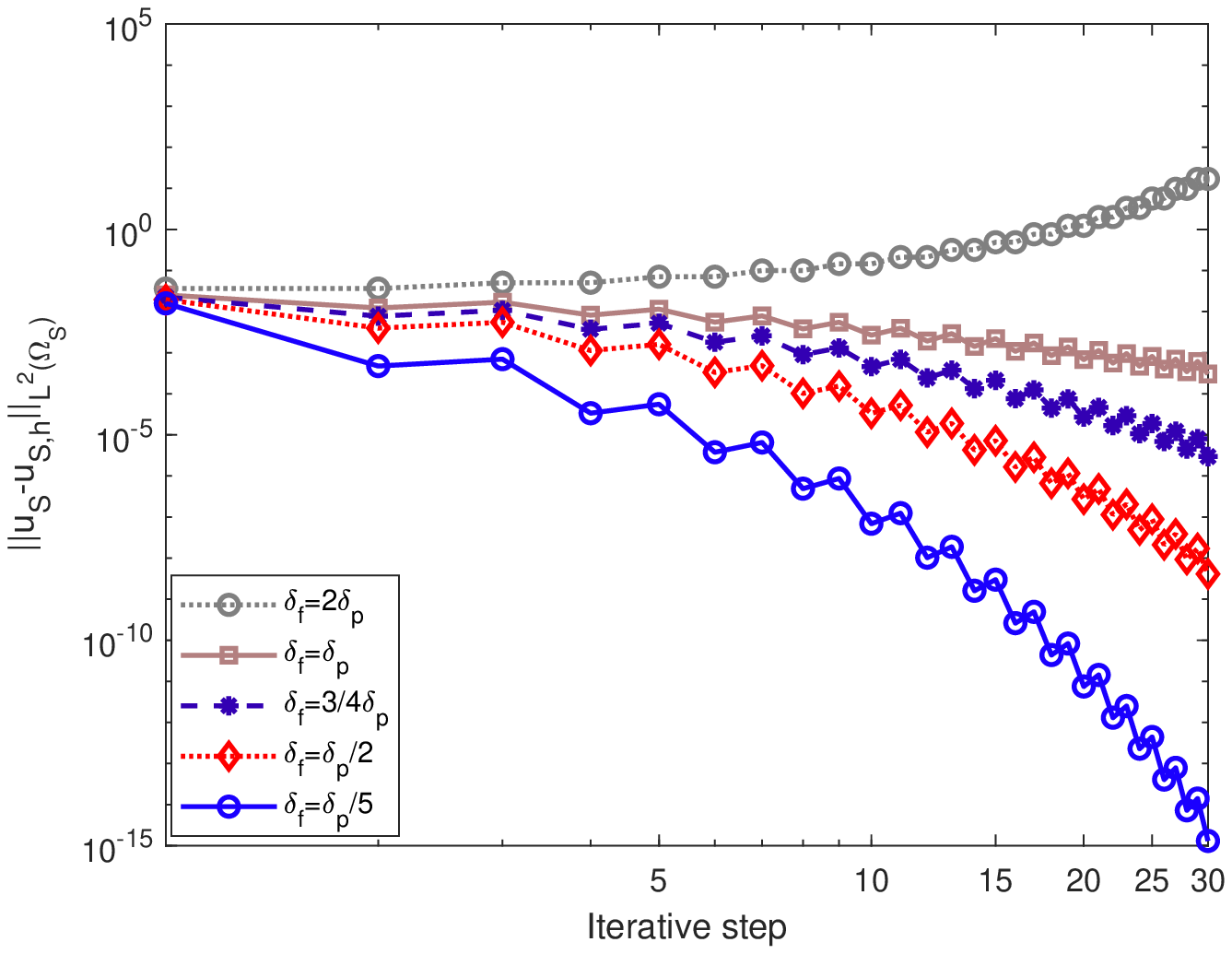}
\includegraphics[width=0.45\textwidth]{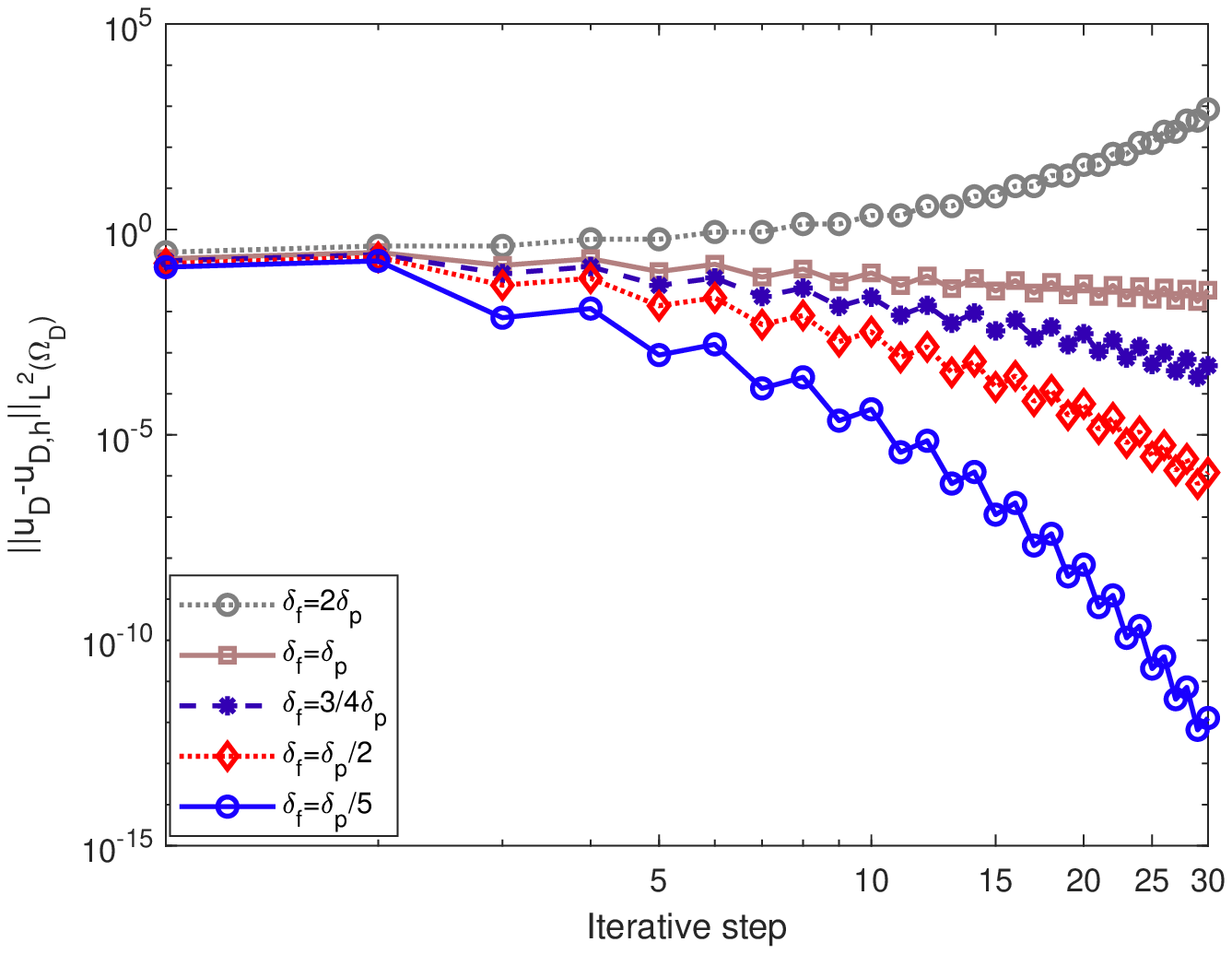}
\caption{The numerical velocity errors for the Stokes (left) and Darcy (right) regions with meshsize $h=1/2^4$ for Example~\ref{ex2}.}
\label{ex2:fig}
\end{figure}

\subsection{Example 3}\label{ex3}

In this example, we set $\Omega_D=(0,1)\times (0,1)$ and $\Omega_S=(0,1)\times (1,2)$.
We use the exact solution defined by
\begin{align*}
 \bm{u}_S=\begin{cases}
x^2(y-1)^2 +  y\\
 -2x(y-1)^3/3
\end{cases},\quad p_D=(2-\pi\sin(\pi x))\cos(\pi y).
\end{align*}
Here, we take $G=1/\mu$ and the interface conditions are satisfied exactly. In this example we attempt to test the convergence of Algorithm DDM with respect to different values of $\mu$. Figure~\ref{ex3:fig} shows the $L^2$ errors of fluid velocity and Darcy velocity with respect to various values of $\mu$ under the setting $\delta_p=\mu,\delta_f=\mu/4$. One can clearly observe that Algorithm DDM converges for all the cases, and converges slower for smaller values of $\mu$. Next, we fix $\mu=1$ and take various combinations of $\delta_f$ and $\delta_p$. We can observe from Figure~\ref{ex3:fig2} that when the ratio $\frac{\delta_f}{\delta_p}$ gets smaller, Algorithm DDM tends to converge faster, which is consistent with our theory.

%In order to set the interface conditions for various values of $\mu$, we set $G=1/\mu$.
%
%In this test, we test the convergence of the iterative method with respect to $\mu$.

\begin{figure}[t]
\centering
\includegraphics[width=0.45\textwidth]{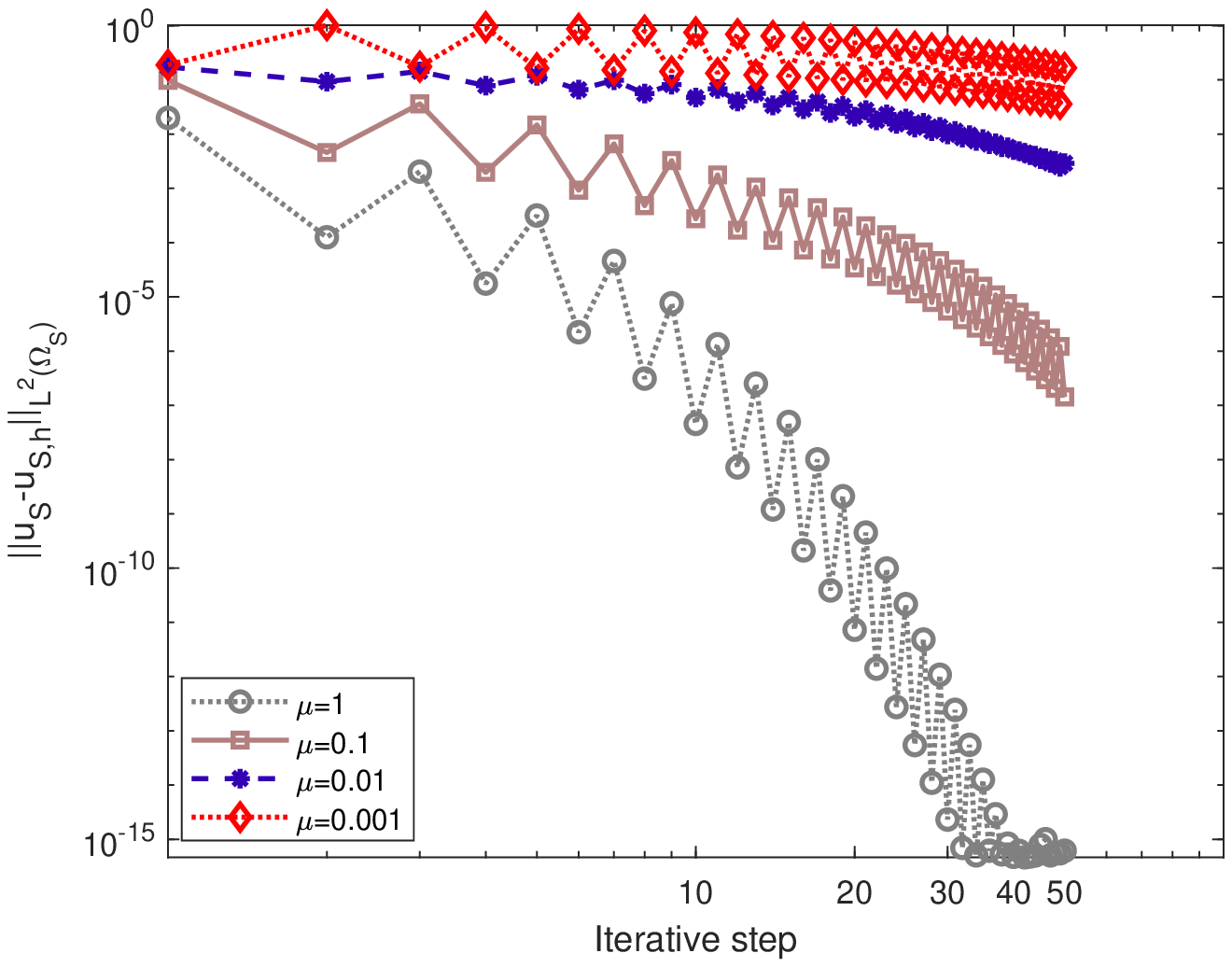}
\includegraphics[width=0.45\textwidth]{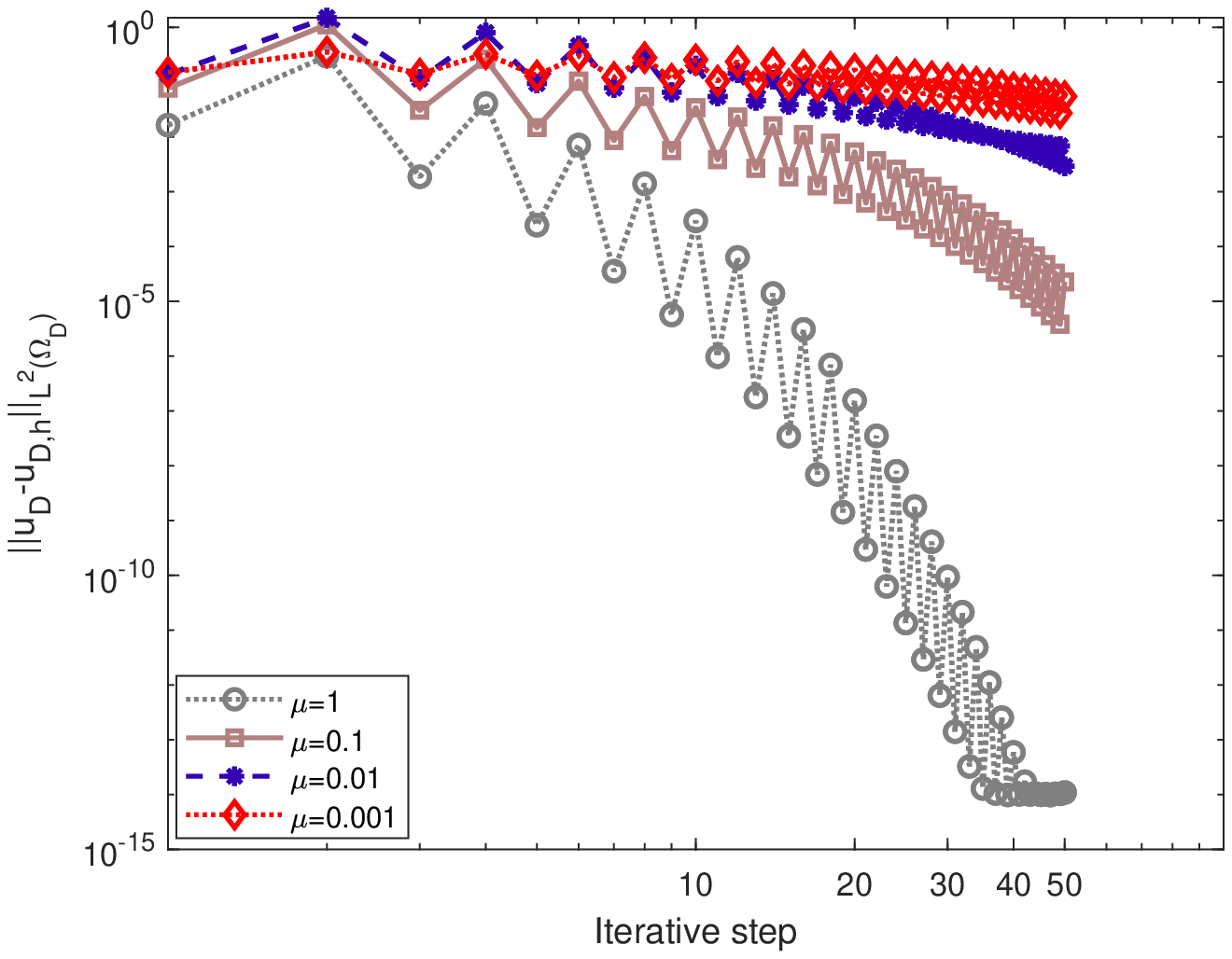}
\caption{The numerical velocity errors for the Stokes (left) and Darcy (right) regions with meshsize $h=1/2^4,\delta_p=\mu,\delta_f=\mu/4$ for Example~\ref{ex3}.}
\label{ex3:fig}
\end{figure}

%Next, we test the convergence of the iterative method for $\mu=1$ with different choices of $\delta_p$ and $\delta_f$. Hhre we fix $\delta_p=\mu$

\begin{figure}[t]
\centering
\includegraphics[width=0.45\textwidth]{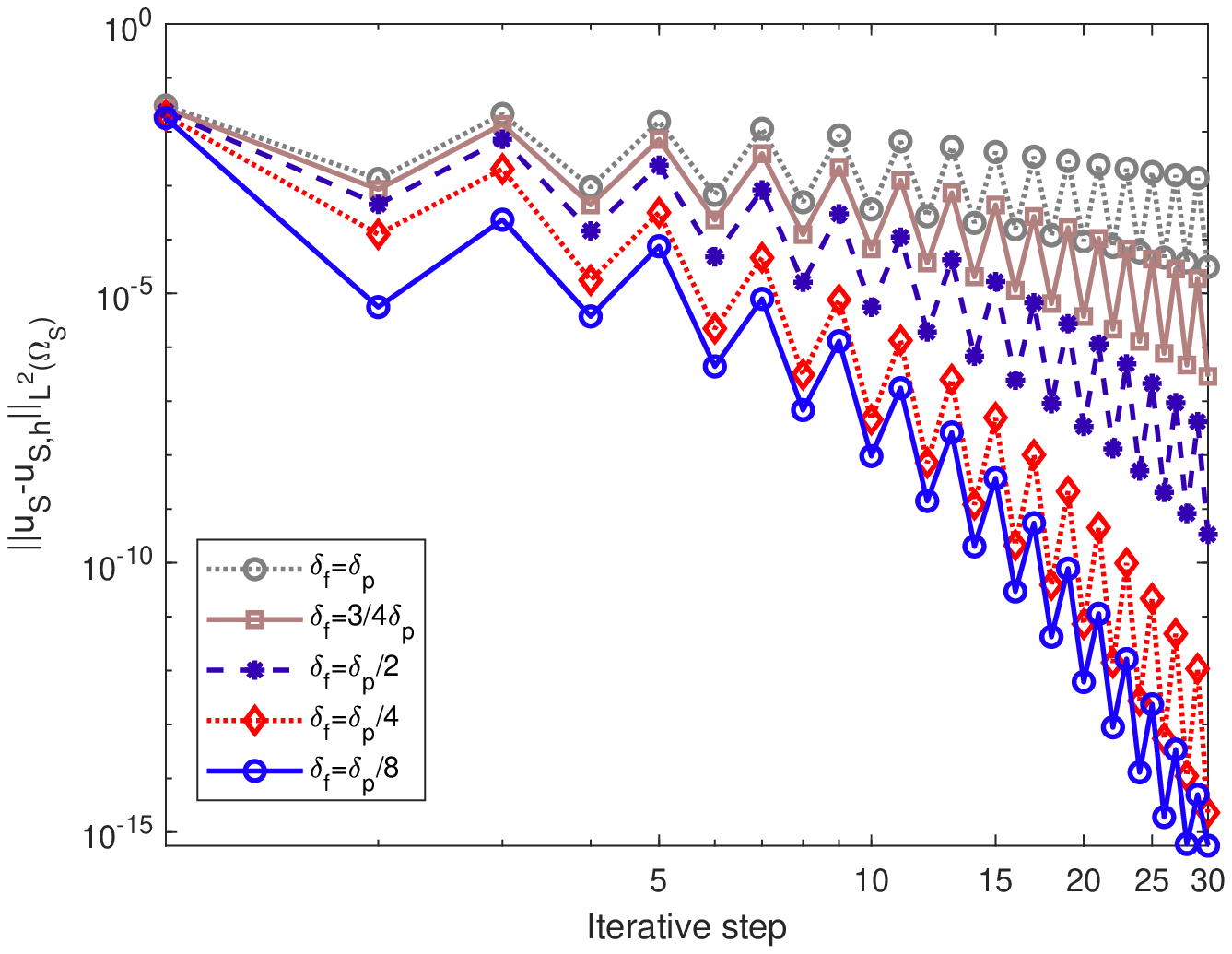}
\includegraphics[width=0.45\textwidth]{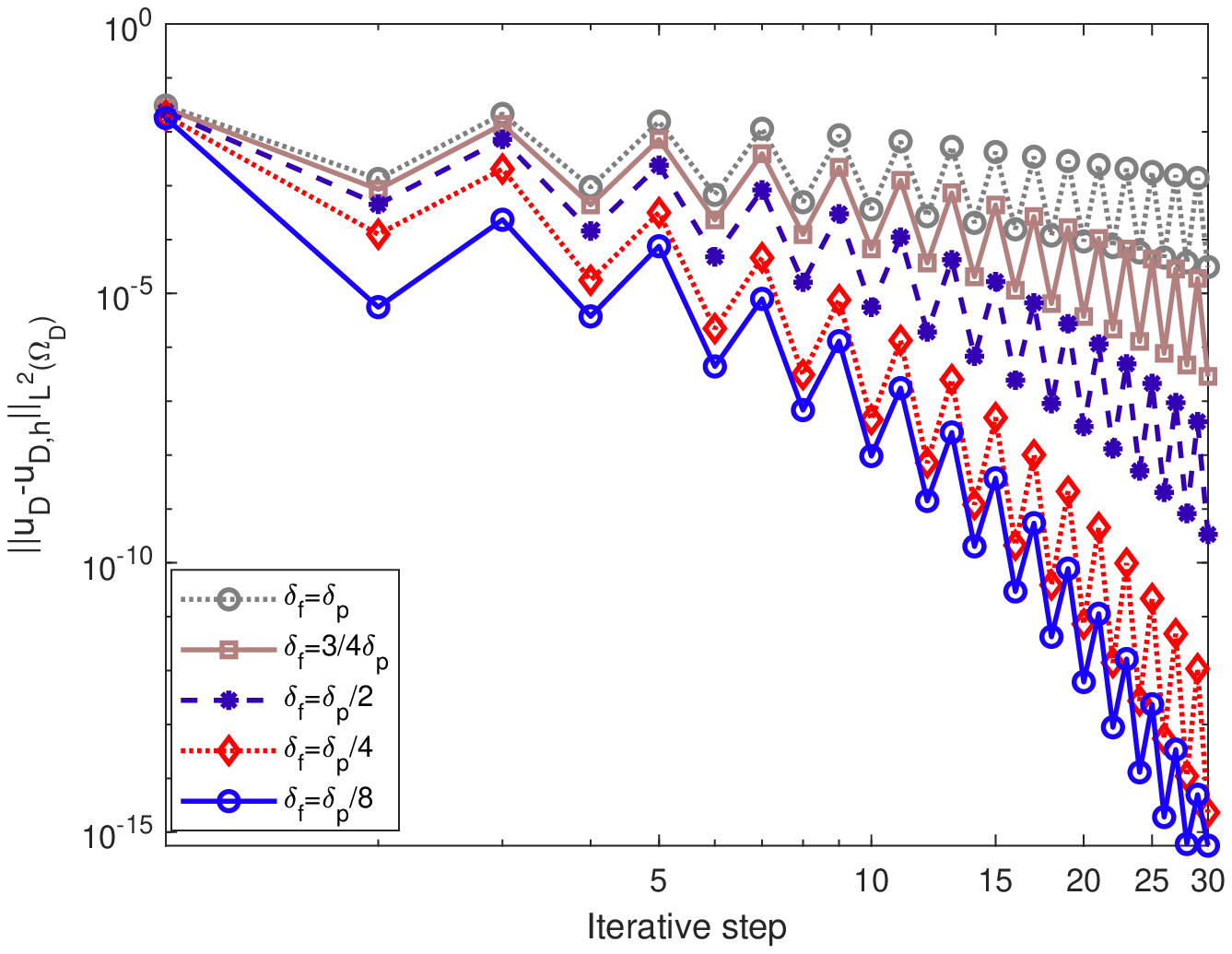}
\caption{The numerical velocity errors for the Stokes (left) and Darcy (right) with meshsize $h=1/2^4$ for Example~\ref{ex3}.}
\label{ex3:fig2}
\end{figure}

%\subsection{}
%
%
%In this test, we test the convergence of the iterative method with respect to $\nu$.
%
%
%\begin{align*}
% \bm{u}_S=\begin{cases}
%x^2(y-1)^2 +  y\\
% -2x(y-1)^3/3
%\end{cases},\quad p_D=(y-1)^3.
%\end{align*}
%which satisfies the interface conditions for various choices of $G$ and $\nu$.

\subsection{Example 4}\label{ex4}

Finally, we use the modified driven cavity flow to test the performance of Algorithm DDM. To this end, we set $\Omega_D=(0,1)\times (0.25,1)$ and $\Omega_S=(0,1)\times (1,1.25)$. The exact solution is unknown.
The boundary conditions for $\bm{u}_S$ are defined as
\begin{align*}
\bm{u}_S&=[\sin(\pi x),0]\quad \mbox{on}\; (0,1)\times \{1.25\},\\
\bm{u}_S&=[0,0]\quad \mbox{on}\; \{0\}\times (1,1.25)\cup \{1\}\times (1,1.25).
\end{align*}
Homogeneous Dirichlet boundary condition is imposed for $p_D$ on $\Gamma_D$.
The source term is taken to be $\bm{f}_S=(0,0)$ and $f_D=2\sin(\pi x)$.

We first display the approximated solution for $G=\mu=1$ with $h=1/32$ in Figure~\ref{ex4:fig1}. The converge of Algorithm DDM for the Stokes and Darcy regions is shown in Figure~\ref{ex4:fig2}, where different values of $\frac{\delta_f}{\delta_p}$ are used. As expected, the algorithm converges for $\delta_f\leq\delta_p$ and it tends to converge faster for smaller values of $\frac{\delta_f}{\delta_p}$.
\begin{figure}[t]
\centering
\includegraphics[width=0.32\textwidth]{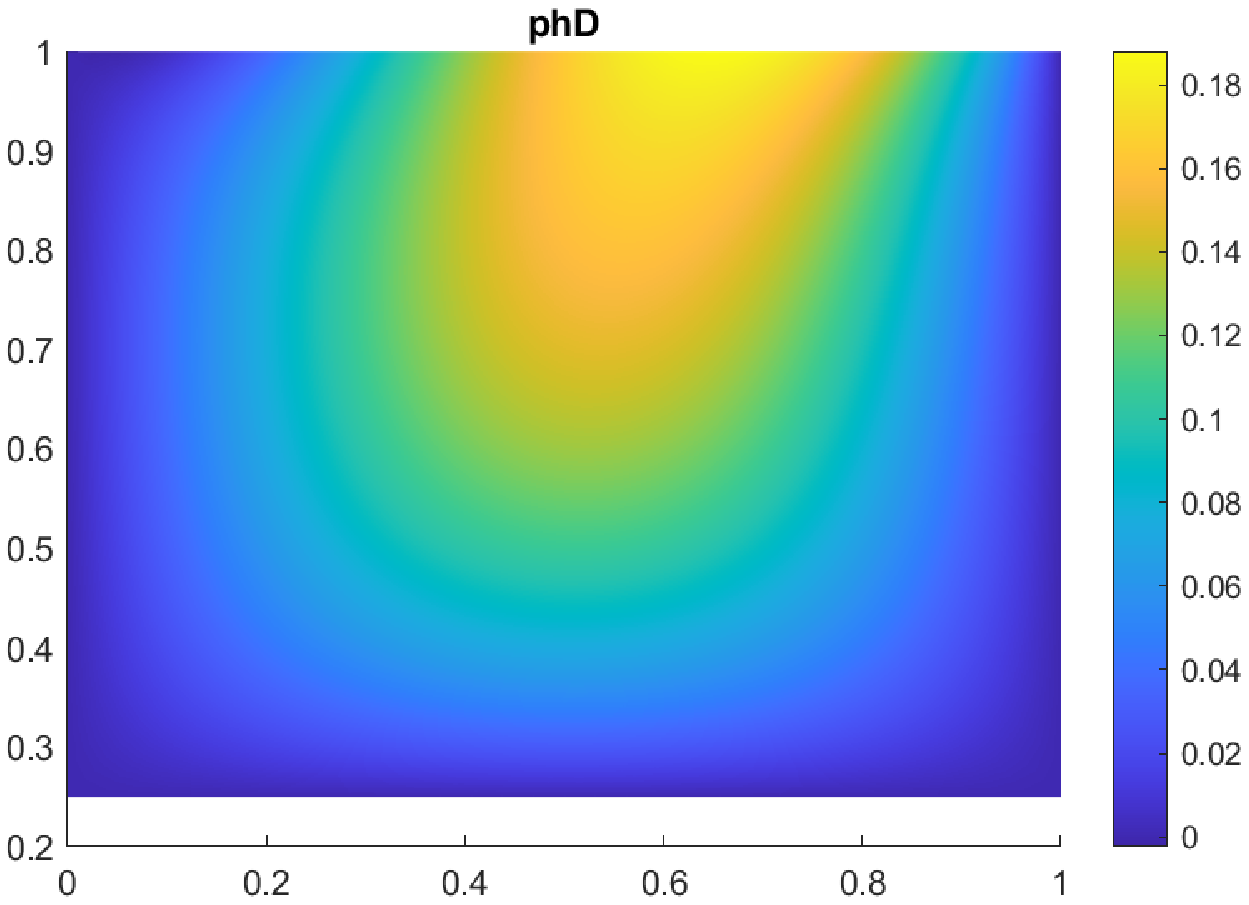}
\includegraphics[width=0.32\textwidth]{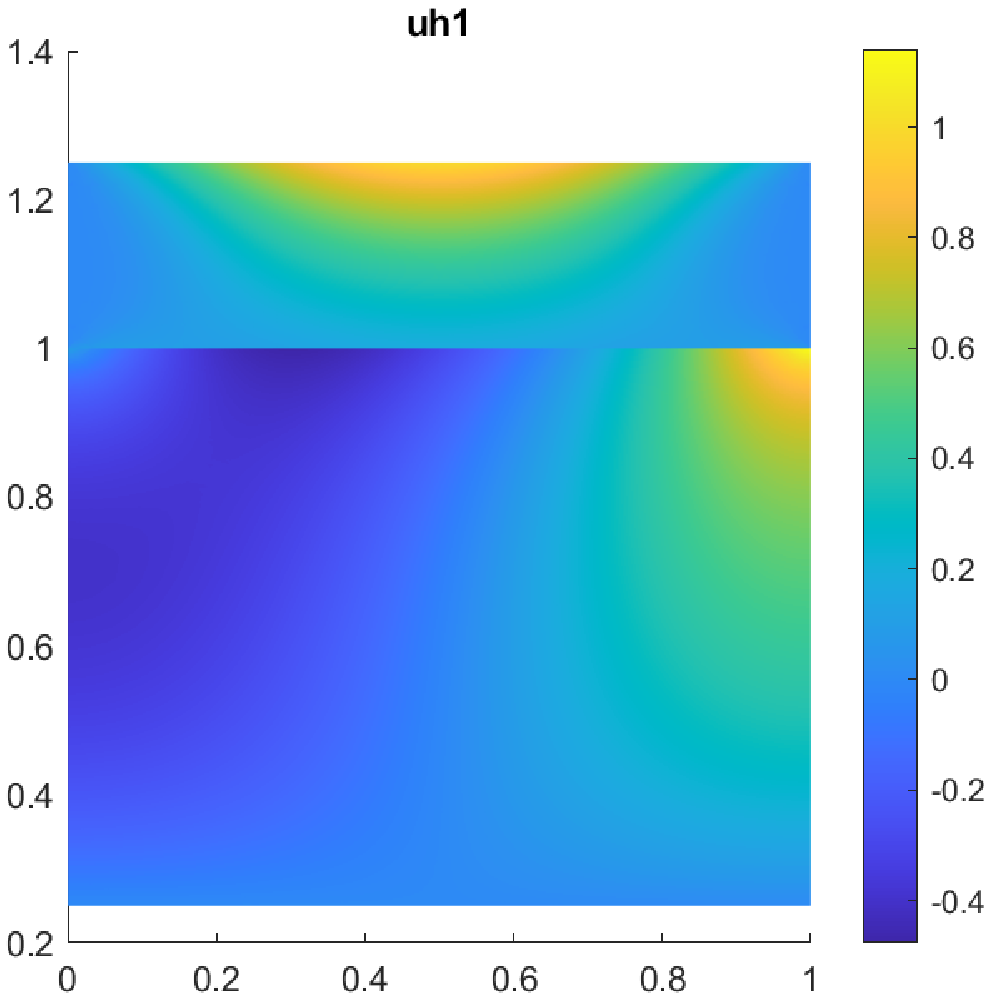}
\includegraphics[width=0.32\textwidth]{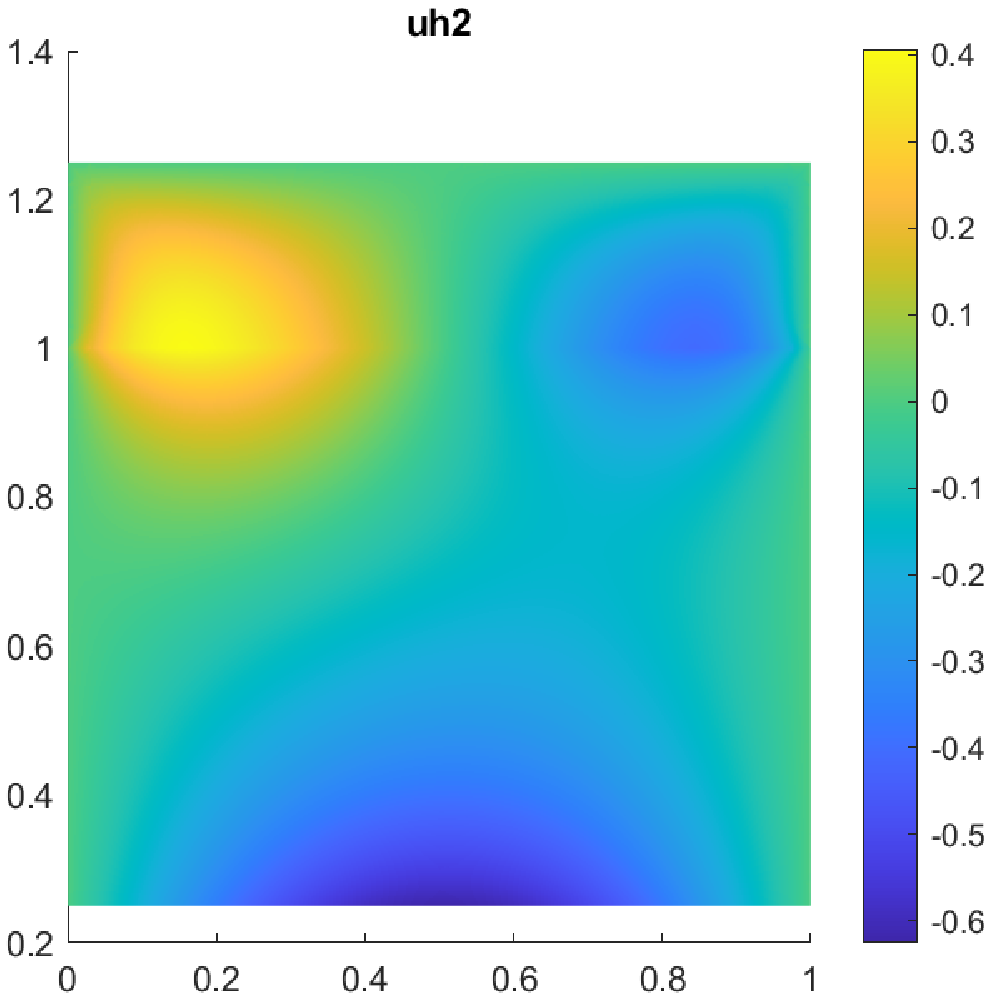}
\caption{Profiles of numerical solution for Example~\ref{ex4}.}
\label{ex4:fig1}
\end{figure}
%Next, we show the errors with respect to difference choices of $\delta_f$ and $\delta_p$ for fixed $\mu=1$.
\begin{figure}[t]
\centering
\includegraphics[width=0.45\textwidth]{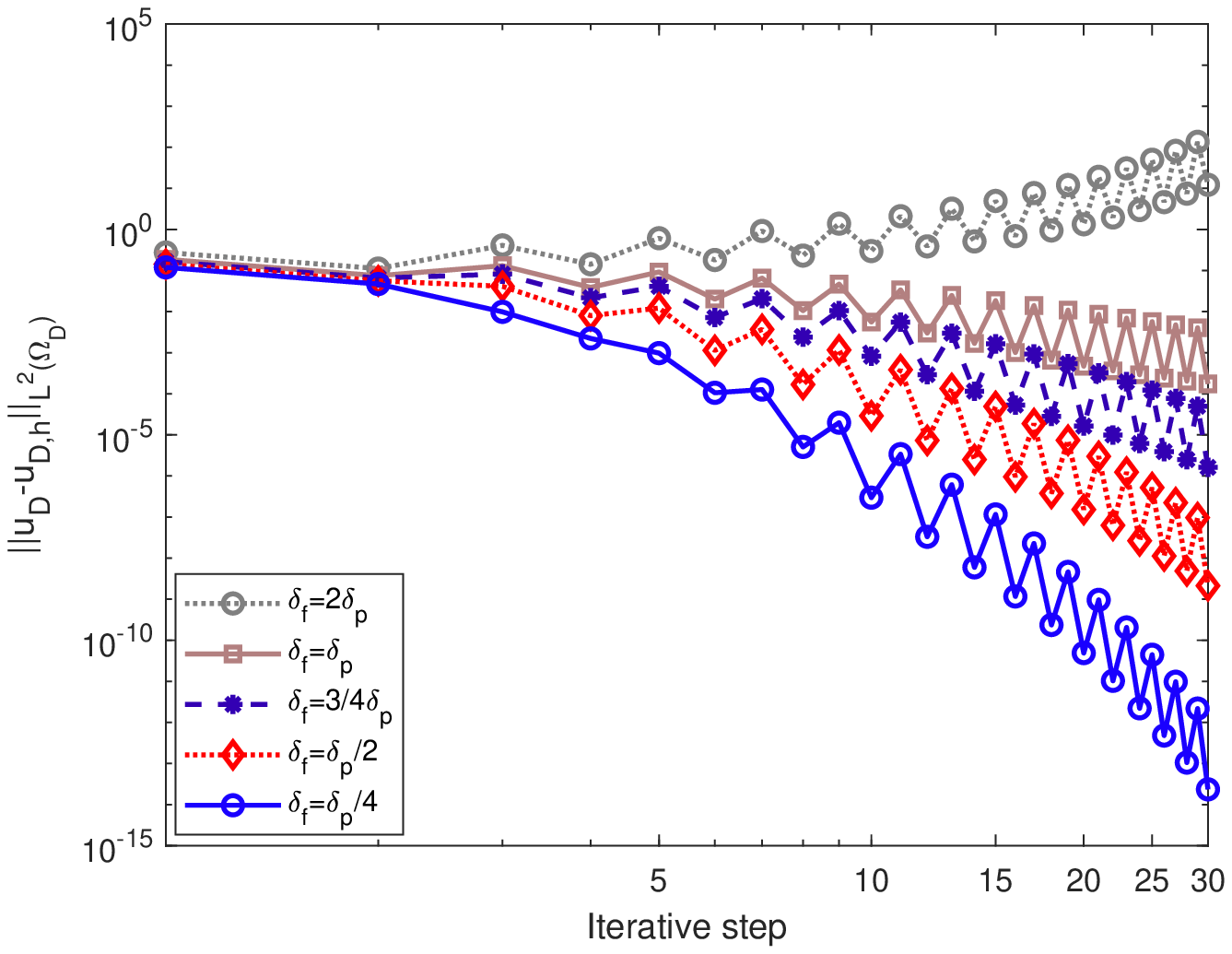}
\includegraphics[width=0.45\textwidth]{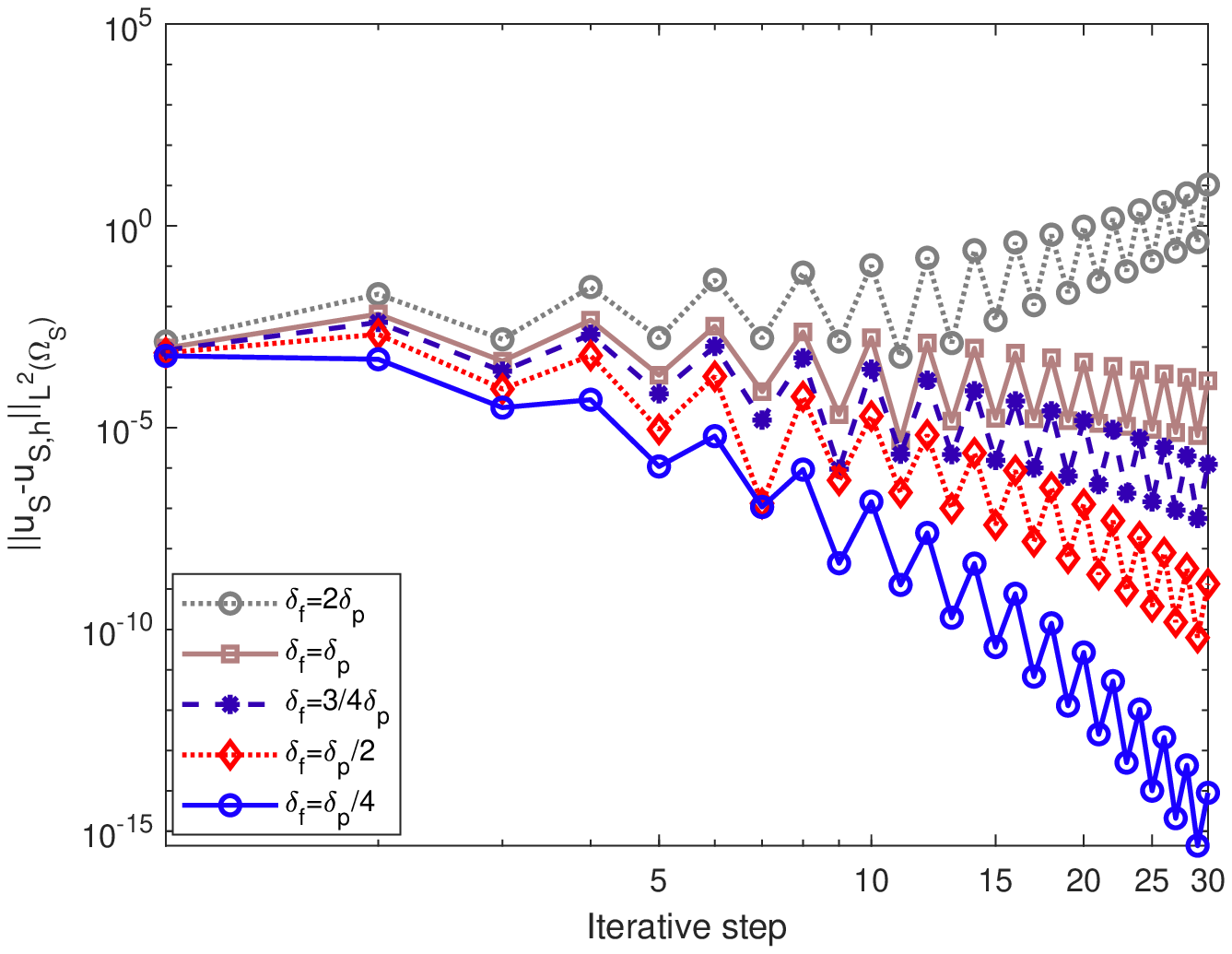}
\caption{The numerical velocity errors for the Stokes (left) and Darcy (right) regions for Example~\ref{ex4}.}
\label{ex4:fig2}
\end{figure}

\section{Conclusion}\label{sec:conclusion}

Our contributions for this paper are twofold. First, we devise and analyze a new method for the coupled Stokes-Darcy problem. This method is based on a stress-velocity formulation for the Stokes equations, which is rarely explored for the coupled Stokes-Darcy problem. The stress-velocity formulation is very important and addresses variables of physical interest. The interface conditions are imposed straightforwardly without resorting to additional variables. In addition, the normal continuity of velocity is satisfied exactly at the discrete level. The convergence error estimates for all the variables are provided. Next, we design a domain decomposition method to decouple the global system into the Stokes subproblem and the Darcy subproblem via a suitable application of Robin-type interface conditions. Moreover, the convergence of the proposed iterative method is analyzed.

%\bibliographystyle{plain}
%\bibliography{reference}

\end{document}